\documentclass[11pt
]{article}
\usepackage[letterpaper,margin=1in,footskip=0.25in]{geometry}
\usepackage{microtype}
\usepackage{cgstix}
\usepackage{mathtools, amssymb, amsthm}
\PassOptionsToPackage{dvipsnames, hyperref}{xcolor}
\usepackage{tikz-cd, extarrows}
\usetikzlibrary{decorations.pathmorphing}

\usepackage[shortlabels]{enumitem}
\setlist{nosep}
\setlist[enumerate]{label=\((\roman*)\),ref=(\roman*)}
\usepackage[style=alphabetic,url,doi,sortcites]{biblatex}
\usepackage{hyperref}
\usepackage[nameinlink]{cleveref}
\hypersetup{
  colorlinks=true,
  linkcolor={blue!50!black},
  citecolor={green!50!black},
  urlcolor={magenta!80!black},
  bookmarks=true,
  bookmarksdepth=2,
  pdftitle = {Higher direct images of ideal sheaves},
  pdfauthor = {Charles Godfrey}
}
\usepackage[nomargin,marginclue,inline]{fixme}

\newcommand{\todo}[1]{\fxnote{#1}}

\usepackage{graphicx}
\usepackage{rotating}
\usepackage{fancyhdr}

\newcommand{\field}[1]{\mathbb #1}

\newcommand{\CC}{\field C}

\newcommand{\QQ}{\field Q}
\newcommand{\NN}{\field N}

\newcommand{\sHom}{\mathcal{H}om}

\newcommand{\sTor}{\mathcal T\!or}

\newcommand{\strshf}[1]{\mathscr{O}_{#1}}
\DeclareMathOperator{\Spec}{Spec}

\newcommand{\PP}{\field P}
\DeclareMathOperator{\Ex}{Ex}
\renewcommand{\AA}{\field A}

\DeclareMathOperator{\pr}{pr}

\DeclareMathOperator{\skel}{sk}

\DeclareMathOperator{\Hom}{Hom}

\DeclareMathOperator{\Sym}{Sym}
\DeclareMathOperator{\Sing}{Sing}
\DeclareMathOperator{\Reg}{Reg}

\DeclareMathOperator{\Tot}{Tot}

\newcommand{\Sch}{\mathbf{Sch}}
\newcommand{\Ab}{\mathbf{Ab}}

\newcommand{\inj}{\hookrightarrow}
\newcommand{\isom}{\xrightarrow{\sim}}

\newcommand{\cC}{\mathcal{C}}
\newcommand{\cD}{\mathcal{D}}

\newcommand{\sE}{\mathscr{E}}
\newcommand{\sF}{\mathscr{F}}
\newcommand{\sG}{\mathscr{G}}

\newcommand{\sI}{\mathscr{I}}
\newcommand{\sJ}{\mathscr{J}}
\newcommand{\sK}{\mathscr{K}}
\newcommand{\sL}{\mathscr{L}}

\newcommand{\sO}{\mathscr{O}}

\DeclareMathOperator{\Bl}{Bl}
\DeclareMathOperator{\snc}{snc}

\DeclareMathOperator{\CM}{CM}

\DeclareMathOperator{\cone}{cone}
\DeclareMathOperator{\Eq}{Eq}

\newcommand{\lsp}[1]{(#1, \Delta_{#1})}

\newcommand{\dcx}[1]{\mathcal{D}(#1)}

\newcommand{\Dbcoh}[1]{D_{\mathrm{coh}}^b(#1)}

\newcommand{\lstrshf}[1]{\mathscr{O}_{#1}(-\Delta_{#1})}
\newcommand{\lcanshf}[1]{\omega_{#1}(\Delta_{#1})}

\numberwithin{equation}{section}

\theoremstyle{plain}
\newtheorem{theorem}[equation]{Theorem}
\crefname{theorem}{Theorem}{Theorems}
\newtheorem*{theorem*}{Theorem}
\newtheorem{lemma}[equation]{Lemma}
\crefname{lemma}{Lemma}{Lemmas}

\crefname{proposition}{Proposition}{Propositions}
\newtheorem{conjecture}[equation]{Conjecture}
\crefname{conjecture}{Conjecture}{Conjectures}
\newtheorem{corollary}[equation]{Corollary}
\crefname{corollary}{Corollary}{Corollaries}

\crefname{problem}{Problem}{Problems}

\theoremstyle{plain}

\newtheoremstyle{named}{}{}{\itshape}{}{\bfseries}{.}{.5em}{#1 \thmnote{#3}}
\theoremstyle{named}

\theoremstyle{definition}
\newtheorem{definition}[equation]{Definition}
\crefname{definition}{Definition}{Definitions}

\crefname{variant}{Variant}{Variants}

\crefname{notation}{Notation}{Notations}

\crefname{convention}{Convention}{Conventions}
\newtheorem{claim}[equation]{Claim}
\crefname{claim}{Claim}{Claims}
\newtheorem{slogan}[equation]{Slogan}
\crefname{slogan}{Slogan}{Slogans}

\crefname{fact}{Fact}{Facts}

\crefname{assumption}{Assumption}{Assumptions}

\crefname{hypothesis}{Hypothesis}{Hypotheses}

\crefname{construction}{Construction}{Constructions}

\crefname{calculation}{Calculation}{Calculations}

\theoremstyle{remark}
\newtheorem{remark}[equation]{Remark}
\crefname{remark}{Remark}{Remarks}

\crefname{observation}{Observation}{Observations}
\newtheorem{example}[equation]{Example}
\crefname{example}{Example}{Examples}
\newtheorem*{example*}{Example}

\crefname{question}{Question}{Questions}
\newtheorem*{question*}{Question}

\crefname{warning}{Warning}{Warnings}
\crefname{enumi}{}{}
\crefname{figure}{Figure}{Figure}

\title{Higher direct images of snc ideal sheaves}
\author{Charles Godfrey\\
\href{mailto:godfrey.cw@gmail.com}{\texttt{godfrey.cw@gmail.com}}\\
}
\date{\today} 

\begin{document}
\maketitle
\thispagestyle{fancy}
\renewcommand{\headrule}{}
\renewcommand{\footrulewidth}{0.5pt}
\fancyhead[L,C,R]{}
\fancyfoot[L]{\footnotesize This work was completed while the
  author was a PhD student in the University of Washington Department of Mathematics. The author was partially supported by the University of
  Washington Department of Mathematics Graduate Research Fellowship, and by the
  NSF grant DMS-1440140, administered by the Mathematical Sciences  Research
  Institute, while in residence at MSRI during the program Birational Geometry
  and Moduli Spaces.}
\fancyfoot[C,R]{}
\abstract{%
  We prove invariance results for the cohomology groups of ideal sheaves of simple
  normal crossing divisors under (a restricted class of) birational morphisms of
  pairs in arbitrary characteristic, assuming a conjecture regarding the existence
  of normal Cohen-Macaulayfications. As an application, we extend some
  foundational results in the theory of rational pairs that were previously known
  only in characteristic 0.
}

\listoffixmes

\todo{Need to address the normality issues with SK's ratl. sings, add the conjecture on normal Macaulayfications DONE. Point out cases where the conjecture is known (e.g. char 0) DONE. Point out that in char 0 our thm. on thrifty proper birational equivalences is still new DONE, and we give a new proof of the ``all for one'' thing (already DONE).}

\section{Introduction}

A foundational problem in birational geometry, posed by Grothendieck in his 1958
ICM address \cite[Problem B]{MR0130879}, asked whether for every proper
birational morphism of non-singular projective varieties \(f\colon X \to Y\),
\[R^i f_* \strshf{X} = 0 \text{  for  } i>0. \]
In characteristic 0 this was answered
affirmatively by Hironaka as a corollary of resolution of singularities \cite[\S
  7 Cor. 2]{MR0199184}.
In characteristic \(p >0\), where resolutions of singularities are not known to
exist, answering Grothendieck's question proved much harder, remaining open
until 2011 when Chatzistamatiou and R\"ulling proved the following theorem.

\begin{theorem}[{\cite[Thm. 3.2.8]{MR2923726}, see also \cite[Thm. 1.1]{MR3427575}, \cite[Thm. 1.19]{kovacsRationalSingularities2022}}]
  \label{thm:cr-hdi-struct}
  Let \(k\) be a perfect field and let \(S\) be a scheme. Suppose \(X\) and \(Y\) are two separated
  finite type \(S\)-schemes which are
  \begin{enumerate}
    \item smooth over \(k\) and
    \item \textbf{properly birational} over \(S\) in the sense that
          there is a commutative diagram
          \begin{equation}
            \label{eq:1}
            \begin{tikzcd}
              & Z \arrow[dl, "r"'] \arrow[dr, "s"]  & \\
              X \arrow[dr, "f"']  \arrow[rr, phantom, "\circlearrowleft"] &   & Y \arrow[dl, "g"] \\
              & S & \\
            \end{tikzcd}
          \end{equation}
          with \(r\) and \(s\) proper birational morphisms.
  \end{enumerate}
  Then, there are isomorphisms of sheaves
  \begin{equation}
    \label{eq:cr-correspondence}
    R^{i} f_{*} \sO_X \isom  R^{i}
    g_{*} \sO_Y  \text{  and  } R^{i} f_{*} \omega_X \isom  R^{i}
    g_{*} \omega_Y \text{  for all  } i.
  \end{equation}
\end{theorem}
One of the primary applications of \cref{thm:cr-hdi-struct} was to extend
foundational results on rational singularities from characteristic 0 to
arbitrary characteristic (for definitions of rational resolutions and
rational singularities see \cref{def:ratl-res}).
\begin{corollary}[{\cite[Cor. 3.2.10]{MR2923726}, see also \cite[Thm. 1.4]{kovacsRationalSingularities2022}}]
  \label{cor:rat-res-all41}
  If \(S\) has a rational resolution, then every resolution of \(S\)
  is rational.
\end{corollary}

In this article, we prove analogues of \cref{thm:cr-hdi-struct,cor:rat-res-all41} for pairs.
\begin{definition}[{slightly more general version of \cite[Def.
          1.5]{MR3057950}}]
  \label{def:pair}
  In what follows a \textbf{pair \(\lsp{X}\)} will mean a reduced,
  equidimensional excellent scheme \(X\) admitting a dualizing complex together
  with a \(\QQ\)-Weil divisor \(\Delta_{X} = \sum_i a_i D_i\) on \(X\) such that
  no irreducible component \(D_i \) of \(\Delta_X\) is contained in
  \(\Sing(X)\).
\end{definition}
\begin{definition}
  \label{def:snc}
  A \emph{simple normal crossing pair} is an equidimensional, regular excellent
  scheme \(X\) together with a reduced effective divisor \(\Delta_X = \sum_i
  D_i\)  such that for every subset \(J \subseteq \{1, \dots , N \}\) the
  scheme-theoretic intersection
  \[ D_J := \cap_{j \in J} D_j \subseteq X \]
  is regular of codimension \(|J|\).
\end{definition}
\begin{remark}
  If \(X\) is regular as in \cref{def:snc} then it admits a dualizing complex.
  By an amazing result of Kawasaki \cite[Cor. 1.4]{MR1859029}, a noetherian ring
  admits a dualizing complex if and only if it is a homomorphic image of a
  finite-dimensional Gorenstein ring.
\end{remark}

As observed in \cite[\S 2.5]{MR3057950}, to generalize \cref{cor:rat-res-all41}
to pairs we must restrict attention to a special class of \emph{thrifty
  resolutions}.

\begin{definition}
  \label{def:abstract-strata-intro}
  A \emph{stratum} of a simple normal crossing pair \((X, \Delta_X = \sum_i D_i)\) is a connected
  (equivalently, irreducible) component of an intersection \(D_J = \cap_{j \in J}D_j\).
\end{definition}

\begin{definition}[
    {compare with \cite[Def. 2.79-2.80]{MR3057950}, \cite[\S 1,
          discussion before Def. 10]{MR3539921}}
  ]
  \label{def:abstract-thriftiness-intro}
  Let \((S, \Delta_S = \sum_i D_i)\) be a pair,
  and assume
  \(\Delta_S\) is reduced and effective. A separated, finite type birational
  morphism \(f: X \to S\) is \emph{thrifty with respect to} \(\Delta_S \) if
  and only if
  \begin{enumerate}
    \item \label{item:over-gen-pts-intro} \(f\) is an isomorphism over the
          generic point of every stratum of \(\snc(S, \Delta_S)\) and
    \item \label{item:at-gen-pts-intro} letting \(\tilde{D}_i = f^{-1}_*D_i\)
          for \(i = 1, \dots, N\) be the strict transforms of the divisors \(D_i\),
          and setting \(\Delta_X := \sum_i \tilde{D}_i\), the map \(f\) is an
          isomorphism at the generic point of every stratum of \(\snc(X,
          \Delta_X)\).
  \end{enumerate}
\end{definition}

Philosophically, a resolution \(\lsp{X} \to \lsp{S}\) is thrifty if it induces a one-to-one birational correspondence
between the \emph{strata} of \(\lsp{X} \) and \( \lsp{S}\) , thus preserving the
combinatorial geometry of the pair \(\lsp{S}\).

In order to prove our main result, \cref{thm:hdi-log-struct}, we need to assume
a conjecture regarding the existence of \emph{normal} Cohen-Macaulayfications
(\cref{conj:macify} below). To contextualize this conjecture we first recall
that requiring normality, the existence of Cohen-Macaulayfications is known.
\begin{theorem}[{\cite[Thm.
          1.6]{CesMac}, see also \cite[Thm.
          1.1]{MR1707481}}]
  \label{thm:macify}
  For every CM-quasi-excellent noetherian scheme \(X\) there exists a
  projective birational morphism \(\pi : \tilde{X} \to X\) such that
  \(\tilde{X}\) is Cohen-Macaulay and \(\pi\) is an isomorphism over the
  Cohen-Macaulay locus \(\CM(X)\subset X\).
\end{theorem}
The notion of CM-quasi-excellence is a weakening of excellence introduced by
\v{C}esnavi\v{c}ius, so in particular the theorem applies to excellent
noetherian schemes. To prove our results, we require the existence of a
\(\tilde{X}\) as appearing in \cref{thm:macify} that is also normal, and it is
at least
technically very useful for us to require that the associated projective
birational morphism \(\pi \) is an isomorphism over the \emph{regular} locus
\(\mathrm{Reg}(X) \).\footnote{It is however possible that overhauling our
  proofs would allow for relaxing this condition.} Note that this is weaker than requiring \(\pi \) to be an
isomorphism over \( \CM(X) \).
\begin{conjecture}[{see also \cite[Conj. 1.1]{CesMac}, \cite[Conj.
          1.14]{kovacsRationalSingularities2022}}]
  \label{conj:macify}
  For every CM-quasi-excellent noetherian scheme \(X\) there exists a
  projective birational morphism \(\pi : \tilde{X} \to X\) such that
  \(\tilde{X}\) is Cohen-Macaulay and normal and \(\pi\) is an isomorphism over the
  regular locus \(\mathrm{Reg}(X) \subset X\).
\end{conjecture}
Due to the existence of resolution of singularities
\cite{MR0199184,temkinDesingularizationQuasiexcellentSchemes2008}, \cref{conj:macify} is known in characteristic zero. If we drop the requirement that
\(\pi\) be an isomorphism over the regular locus \(\mathrm{Reg}(X) \subset X\),
\cref{conj:macify} is known for varieties of dimension at most 4 over algebraically closed
fields by \cite[Cor. 1.8]{brodmannTwoTypesBirational1983a}. To the best of our
knowledge, \cref{thm:hdi-log-struct} is new even in characteristic 0.\footnote{Although
  it can be proved there without too much difficulty using the machinery
  developed in \cite[\S 2.5]{MR3057950}, in particular without using the
  techniques developed in this paper.}

\todo{Note that this has  different hypotheses on the base scheme than CR11. Probably one can reduce from their arbitrary \(S \) to this excellent Noetherian one by Noetherian approximation?? Should point this out.}
\begin{theorem}[{\cref{cor:complex-on-Z}}]
  \label{thm:hdi-log-struct}
  Let \(S\) be an excellent noetherian scheme and let \(\lsp{X}\) and
  \(\lsp{Y}\) be simple normal crossing pairs separated and of  finite type over
  \(S\). Suppose \(\lsp{X}\) and \( \lsp{Y}\) are properly birational over \(S\) in the
  sense that there is a commutative diagram
  \begin{equation}
    \label{eq:4}
    \begin{tikzcd}
      & \lsp{Z} \arrow[dl, "r"'] \arrow[dr, "s"]  & \\
      \lsp{X} \arrow[dr, "f"']  \arrow[rr, phantom, "\circlearrowleft"] &   & \lsp{Y} \arrow[dl, "g"] \\
      & S & \\
    \end{tikzcd}
  \end{equation}
  where \(r\), \(s\) are proper and birational morphisms, and assume \(
  \Delta_{Z} = r_{*}^{-1}\Delta_{X} = s_{*}^{-1} \Delta_{Y} \).
  If \(r\) and \(s\) are \emph{thrifty} and \cref{conj:macify} holds, then there are quasi-isomorphisms
  \begin{equation}
    \label{eq:5}
    Rf_* \lstrshf{X} \simeq Rg_* \lstrshf{Y} \text{  and  } Rf_* \lcanshf{X} \simeq Rg_* \lcanshf{Y}.
  \end{equation}
\end{theorem}

In order to state an analogue of \cref{cor:rat-res-all41}, we say what we
mean by rational singularities of pairs.

\begin{definition}[{compare with \cite[Def. 2.78]{MR3057950}}]
  \label{def:ratl-res-intro}
  Let \(\lsp{S}\) be a pair as in \cref{def:pair} and assume
  \(\Delta_S\) is reduced and effective.  A proper birational morphism \(f: X
  \to S\) is a \emph{rational resolution} if and only if
  \begin{enumerate}
    \item \label{item:ratlsings1-intro} \(X\) is regular and the strict transform
          \(\Delta_X := f^{-1}_* \Delta_S\) has simple normal crossings,
    \item \label{item:ratlsings2-intro} the natural morphism \(\sO_S(-\Delta_S) \to
          Rf_* \sO_X(-\Delta_X)\) is a quasi-isomorphism, and
  \end{enumerate}
  letting \(\omega_X = h^{-\dim X}\omega_X^\bullet \) where we use
  \(\omega_X^\bullet = f^! \omega_S^\bullet\) as a normalized dualizing complex
  on \(X\),
  \begin{enumerate}[resume]
    \item \label{item:ratlsings3-intro} \(R^i f_* \omega_X(\Delta_X) = 0\) for
          \(i>0\).
  \end{enumerate}
  The pair \(\lsp{S}\) is \emph{resolution-rational}  if and only if it
  has a thrifty rational resolution.
\end{definition}
We wish to emphasize that this is not the only definition of rational pairs
available in the literature: Schwede and Takagi adopted a different definition
in \cite{ST08}.
Here we focus on the variant of rational pairs defined in \cite[\S
  2.5]{MR3057950}, simply because it is the one to which our methods most directly
apply, however we view identifying and studying a notion of rational pairs that
simultaneously generalizes those of \cite{ST08} and \cite{MR3057950} as an
interesting question for future work. With \cref{def:ratl-res-intro} in hand,
the precise statement of our result is:

\begin{theorem}[{\cite[Cor. 2.86]{MR3057950} in characteristic 0, \cref{thm:all41} in arbitrary characteristic}]
  \label{thm:ratl-res}
  Let \(\lsp{S}\) be a pair, with \(\Delta_{S}\) reduced and effective. If
  \(\lsp{S}\) has a thrifty rational resolution \(f: \lsp{X} \to \lsp{S}\), and
  if \cref{conj:macify} holds, then
  every thrifty resolution \(g: \lsp{Y} \to \lsp{S}\) is rational.
\end{theorem}

To prove \cref{thm:hdi-log-struct} and \cref{thm:ratl-res} we take a small step
outside of the category of schemes, to that of \emph{semi-simplical schemes}
(for their definitions and basic properties we refer to
\cref{sec:simpres-specseq}). We hope to show that this is a natural approach ---
for example, every pair \(\lsp{X}\) has a naturally associated semi-simplicial
scheme categorifying its set of strata. We emphasize that all of the
semi-simplicial schemes \(X_\bullet\) we consider have strong
finite-dimensionality properties (\(X_i = \emptyset\) for \(i >> 0\)). Some of
our ideas could be of interest even in characteristic 0; for example, we
obtain a criterion for a pair \(\lsp{S}\) to have rational singularities in
terms of a resolution of the associated semi-simplicial scheme \(S_\bullet\):

\begin{lemma}[{\cref{lem:ssimp-ratl-pairs}}]
  \label{lem:ssimp-ratl-pairs-intro}
  Let \(\lsp{S}\) be a pair, with \(\Delta_{S}\) reduced and effective. Let
  \(\epsilon^S: S_\bullet\to S\) be the associated augmented semi-simplicial scheme and
  suppose
  \begin{equation}
    \begin{tikzcd}
      X_\bullet \arrow[r, "f_\bullet"] \arrow[d,"\epsilon^X"] & S_\bullet \arrow[d, "\epsilon^S"] \\
      X \arrow[r, "f"] & S
    \end{tikzcd}
  \end{equation}
  is a resolution. Then, \(\lsp{S}\) is a rational pair if and only if the sheaf
  \(\lstrshf{S}\) is Cohen-Macaulay and the natural map \(\lstrshf{S} \to Rf_*
  \mathscr{K}\) is a quasi-isomorphism, where \(\mathscr{K}\) is naturally
  defined complex on \(X\).\footnote{What exactly ``resolution'' means and how
    \(\mathscr{K}\) is defined will be
    made precise in \cref{lem:ssimp-ratl-pairs}.}
\end{lemma}

A benefit of this lemma is that resolutions of the semi-simplicial scheme
\(S_\bullet\) are arguably \emph{easier} to construct than thrifty resolutions
of pairs: in \cite[Thm. 10.45]{MR3057950}, the existence of thrifty resolutions
was proved using the refined log-resolution theorems of
\cite{MR1322695,MR1440306}. In contrast, resolving semi-simplicial schemes
requires only standard resolution of singularities together with elementary
algorithms for inductively constructing semi-simplicial schemes (quite similar
to the constructions of hyper-resolutions in the theory of Du Bois singularities
\cite{MR498551,MR498552,MR613848}). Of course, the question of which route is
easier is subjective, but in light of \cref{lem:ssimp-ratl-pairs-intro} we think
an overarching "slogan" of this work is:

\begin{slogan}
  Rationality of a pair \((S, \Delta)\) can be tested using any
  resolution \(X_\bullet \to S_\bullet\) of a semi-simplicial scheme
  \(S_\bullet\) associated to \(S\) (coming from closures of strata of the snc locus
  \(\snc(S, \Delta)\)), and thrifty resolutions simply provide one way to
  obtain such resolutions \(X_\bullet \to S_\bullet\).
\end{slogan}

\subsection{Overview}

We next provide some motivation for the appearance of semi-simplicial schemes in
the proof of \cref{thm:hdi-log-struct}. To begin we may translate the condition that a birational morphism
\(f: \lsp{X} \to \lsp{S}\) of simple normal crossing pairs\footnote{We could
  relax the condition that both pairs are snc, but it will make this motivational
  discussion simpler.} with \(\Delta_X = f^{-1}_* \Delta_S\)  is thrifty into the
statement that the \emph{dual complexes} \(\cD(\Delta_X)\) and \(\cD(\Delta_S)\)
are isomorphic. The dual complex \(\cD(\Delta_X)\) is usually described as the
\(\Delta\)-complex (in the sense of \cite[\S 2.1]{MR1867354}) with 0-cells the
irreducible components \(D^X_i\) of \(\Delta_X = \sum_i D^X_i\), 1-cells the
components of intersections \( D^X_i \cap D^X_j \) for \(i<j\) with gluing maps
corresponding to the inclusions \(D^X_i \cap D^X_j \subseteq D^X_i \) and
\(D^X_i \cap D^X_j \subseteq D^X_j\), and so on --- in terms of
\cref{def:abstract-strata-intro}, the cells of \(\cD(\Delta_X)\) correspond to
strata of \(\lsp{X}\), with gluing maps corresponding to inclusions of strata.
The topological properties of \(\cD(\Delta_X)\) have been extensively studied,
for example in this non-exhaustive list of references:
\cite[]{MR3791210,MR3096516,MR2239783,MR0441970}. Upon inspection we see that a
\(\Delta\)-complex is precisely a semi-simplicial set, and that
\(\cD(\Delta_X)\) is the semi-simplical set obtained by taking \(\pi_0\)
(connected components) of a semi-simplicial \emph{scheme} \(X_\bullet\), with
\[X_i = \coprod_{|J| = i+1} (\cap_{j \in J} D^X_j) \text{ for } i \geq 0\]

The thriftyness hypotheses of \cref{thm:hdi-log-struct}  ensure that \(\lsp{X}\)
and \(\lsp{Y}\) have the same dual complex, which provides enough rigidity to
attempt to prove \cref{thm:hdi-log-struct} by induction on  \(\dim X\)  and the
number of components of \(\Delta_X \), using \cref{thm:cr-hdi-struct} as a base case. For example, we have exact sequences
\[0\to \sO_X(-\Delta_X) \to \sO_X(-\Delta_X+D^X_1) \to
  \sO_{D^X_1}(-\Delta_X+D^X_1|_{D^X_1}) \to 0 \] and similarly on \(Y\).
We can even assume by induction the existence of already-defined
quasi-isomorphisms in a diagram
\begin{equation}
  \label{eq:trying-to-fill-in-a-triangle}
  \begin{tikzcd}
    Rf_* \sO_X(-\Delta_X) \arrow[r] \arrow[d, dashed, "\alpha"] &Rf_* \sO_X(-\Delta_X + D_1^X) \arrow[r, "\rho^X"] \arrow[d, "\beta", "\simeq"'] \arrow[dr, phantom, "(\ast)"] &Rf_*  \sO_{D_1^X}(-\Delta_X + D_1^X |_{D_1^X}) \arrow[r] \arrow[d, "\gamma", "\simeq"']&\cdots \\
    Rg_* \sO_Y(-\Delta_Y) \arrow[r] &Rg_* \sO_Y(-\Delta_Y + D_1^Y) \arrow[r, "\rho^Y"] &Rg_*  \sO_{D_1^Y}(-\Delta_Y + D_1^Y |_{D_1^Y}) \arrow[r] &\cdots
  \end{tikzcd}
\end{equation}
\emph{If} the square \((\ast)\) commutes, then using only the fact that
\(\Dbcoh{S}\) is a triangulated category we get a quasi-isomorphism \(\alpha\)
on the dashed arrow. However, in this approach \(\beta, \gamma\) are themselves
defined by induction, and so to know \((\ast)\) commutes we must take one
inductive step further, considering maps of distinguished triangles
\begin{equation}
  \label{eq:dt1}
  {\tiny
    \begin{tikzcd}[column sep=small]
      Rf_* \sO_X(-\Delta_X + D_1^X) \arrow[r] \arrow[d] &Rf_* \sO_X(-\Delta_X + D_1^X +D_2^X) \arrow[r] \arrow[d]  &Rf_*  \sO_{D_2^X}(-\Delta_X + D_1^X |_{D_2^X}) \arrow[r] \arrow[d]&\cdots \\
      Rg_* \sO_Y(-\Delta_Y + D_1^Y) \arrow[r] &Rg_* \sO_Y(-\Delta_Y + D_1^Y +D_2^Y) \arrow[r] &Rg_*  \sO_{D_2^Y}(-\Delta_Y + D_1^Y |_{D_2^Y}) \arrow[r] &\cdots
    \end{tikzcd} \text{  and  }}
\end{equation}
\begin{equation}
  \label{eq:dt2}
  {\tiny
    \begin{tikzcd}[column sep=small]
      Rf_* \sO_{D_1^X}(-\Delta_X + D_1^X |_{D_1^X}) \arrow[r] \arrow[d] &Rf_* \sO_{D_1^X}(-\Delta_X + D_1^X +D_2^X |_{D_1^X}) \arrow[r] \arrow[d]  &Rf_*  \sO_{D_1^X \cap D_2^X}(-\Delta_X + D_1^X +D_2^X |_{D_1^X \cap D_2^X}) \arrow[r] \arrow[d]&\cdots \\
      Rg_* \sO_{D_1^Y}(-\Delta_Y + D_1^Y |_{D_1^Y}) \arrow[r]&Rg_* \sO_{D_1^Y}(-\Delta_Y + D_1^Y +D_2^Y |_{D_1^Y}) \arrow[r] \ &Rg_*  \sO_{D_1^Y \cap D_2^Y}(-\Delta_Y + D_1^Y +D_2^Y |_{D_1^Y \cap D_2^Y}) \arrow[r] &\cdots
    \end{tikzcd}}
\end{equation}
together with a map from \eqref{eq:dt1} to \eqref{eq:dt2} including the square
\((\ast)\), and so on. It is certainly possible that the correct induction
hypothesis (building in not only quasi-isomorphisms like \(\beta, \gamma\) in
\eqref{eq:trying-to-fill-in-a-triangle} but also commutativity hypotheses) and
some careful analysis of diagrams in \(\Dbcoh{S}\) could make this strategy
work, but the author had no such luck.  A separate technical issue the above approach encounters is that at some point
in the base case, we must analyze how the isomorphisms of
\cref{thm:cr-hdi-struct} behave with respect to restrictions, i.e. diagrams of
schemes like
\[
  \begin{tikzcd}
    D_1^X \arrow[d] & D_1^Z \arrow[l] \arrow[r] \arrow[d] & D_1^Y \arrow[d] \\
    X & Z \arrow[l ]\arrow[r] & Y
  \end{tikzcd}
\]
Delving into the methods of
\cite{MR2923726,MR3427575,kovacsRationalSingularities2022}, this analysis runs
into subtle aspects of Grothendieck duality, \emph{especially} since for this
approach to work we do require morphisms in \(\Dbcoh{S}\), not simply of
cohomology sheaves as in \cref{thm:cr-hdi-struct}.

Despite the aforementioned technical issues, what is clear is that this
attempted induction takes place on the semi-simplicial schemes \(X_\bullet\) and
\(Y_\bullet\) underlying the dual complexes \(\cD(\Delta_X)\) and
\(\cD(\Delta_Y)\). Under necessary thriftiness hypotheses, in the situation of
\cref{thm:hdi-log-struct} we find that there is also an auxiliary
semi-simplicial scheme \(Z_\bullet \) together with morphisms \(X_\bullet
\xleftarrow[]{r_\bullet} Z_\bullet \xrightarrow[]{s_\bullet} Y_\bullet\) which
are birational in each simplicial degree. Using the refined forms of Chow's lemma
and resolution of indeterminacies from Conrad's article on Deligne's notes on
Nagata compactifications \cite{MR2356346}, together with the \cref{conj:macify},  we can prove the existence of
such a \(Z_\bullet\) where each scheme \(Z_i\) is Cohen-Macaulay and normal and the
morphisms \(X_i \xleftarrow[]{r_i} Z_i \xrightarrow[]{s_i} Y_i\) are \emph{projective}
--- this occupies \Cref{sec:constructions,sec:snc-thrifty}.  We then make
essential use of recent work of Kov\'acs \cite[Thm. 1.4]{kovacsRationalSingularities2022} to conclude
that there are natural maps \(\sO_{X_i} \to R r_{i*} \sO_{Z_i}\) and \(\sO_{Y_i}
\to Rs_{i*} \sO_{Z_i}\) are quasi-isomorphisms for all \(i\). A more detailed overview of this
construction is included at the beginning of \Cref{sec:constructions}.

The remainder of our proof is pure homological algebra: in
\Cref{sec:simpres-specseq} we show that when \(\lsp{X}\) is a simple normal
crossing pair (more generally, when the components \(D_i^X \) of \(\Delta_X\)
form a \emph{regular sequence}, see \cref{def:reg-seq-div}) the ideal sheaf
\(\sO_X(-\Delta_X)\) admits a \v{C}ech-type resolution of the form
\[ \sO_X(-\Delta_X) \to \sO_X \to \sO_{X_0} \to \sO_{X_1} \to \cdots, \]
in other words we can recover \(\sO_X(-\Delta_X)\) from an augmentation
morphism \(X_\bullet \to X\). Moreover, we can recover the \emph{cohomology} of
\(\sO_X(-\Delta_X)\) from a descent-type spectral sequence
\cref{cor:desc-ss-hdi} --- the last major technical ingredient is a
comparison of the resulting spectral sequences associated to \(X\), \(Y\) and
\(Z\).

\Cref{sec:app-ratl-pairs} deals with applications to rational pairs, in
particular \cref{thm:ratl-res} and \cref{lem:ssimp-ratl-pairs-intro}. \Cref{sec:ex-non-thrifty} includes some
new examples illustrating the subtleties of thrifty and rational resolutions of
pairs, including an affirmative answer to a question of Erickson and Prelli on
whether there exists a non-thrifty rational resolution of a pair \((S, \Delta)\)
--- our \((S, \Delta )\) is even a rational pair, and the resolution is related
to the famous Atiyah flop.

\subsection{Acknowledgements}
I would like to thank Jarod Alper, Chi-yu Cheng, Kristin DeVleming, Gabriel
Dorfsman-Hopkins, Max Lieblich, Takumi Murayama, Karl Schwede and Tuomas Tajakka
for helpful conversations, multiple anonymous reviewers for useful feedback and my advisor S\'andor Kov\'acs for  proposing the
problem of extending results \cref{thm:cr-hdi-struct} to pairs and for many
valuable suggestions.

\section{Regular sequences of divisors and descent spectral sequences}
\label{sec:simpres-specseq}

\subsection{Semi-simplicial schemes and their derived categories}
To any simple normal crossing pair we can naturally associate a
\emph{semi-simplicial scheme}. A primary reference for the theory of
semi-simplicial schemes is \cite[Vbis]{MR0354653}; since many elementary facts
about \emph{simplicial} schemes carry over to semi-semi-simplicial schemes,
\cite{conradCOHOMOLOGICALDESCENT}, \cite[\S 2.4]{MR3495343} and
\cite[\href{https://stacks.math.columbia.edu/tag/0162}{Tag
    0162}]{stacks-project} are also relevant. What follows is a condensed summary of
the machinery we need.

Let \(\Lambda \) denote the category with objects the sets \([i] := \{0, 1, 2,
\dots, i\}\) for \(i \in \NN\) and with morphisms the \emph{strictly increasing}
functions \([j] \to [i]\); in particular \(\Hom_\Lambda([j], [i]) = \emptyset\)
if \(j > i\).\footnote{In \cite[Vbis]{MR0354653} \(\Lambda\) is denoted by
  \(\Delta^+\) so this notation is non-standarad, but seemed necessary due to the
  number of divisors \(\Delta\) and pairs \((X, \Delta)\) considered below. My apologies.}  A \emph{semi-simplicial object} in a category \(\cC\) is a
functor \(\Lambda^{\textup{op}} \to \cC\); semi-simplicial \(\cC\)-objects
naturally form a category, the functor category \(\cC^{\Lambda^{\textup{op}}}\).
Any morphism \(\varphi: [j] \to [i]\) can be written non-uniquely as a
composition of the basic morphisms
\[ \delta^i_k :[i-1] \mapsto [i] \text{ defined by } \delta^i_k(x) =
  \begin{cases}
    x   & \text{if } x < k  \\
    x+1 & \text{ otherwise}
  \end{cases} \]
(so \( \delta^i_k\) skips \(k\))
\cite[\href{https://stacks.math.columbia.edu/tag/0164}{Tag
    0164}]{stacks-project}, and hence a semi-simplicial object \(X :
\Lambda^{\textup{op}} \to \cC\) is equivalent to a sequence of objects \(X_i :=
X([i])\) together with morphisms
\begin{equation}
  \label{eq:ssimp-rel}
  d^i_k := X(\delta^i_k ): X_i \to X_{i-1} \text{ subject to the relations  } d^{i-1}_k \circ d^i_l = d^{i-1}_{l-1} \circ d^i_k,
\end{equation}
and all semi-simplicial objects below will be obtained from such an explicit
description. In what follows semi-simplicial objects will be denoted with a
\(\bullet\), e.g. ``the semi-simplicial scheme \(X_\bullet\)'' (to distinguish
them from plain schemes).

When \(\cC\) is a category of schemes, a \emph{sheaf} on a semi-simplicial
scheme \(X_\bullet\) is the data of a sheaf \(\sF_i\) on each scheme \(X_i\)
together with morphisms of sheaves \(\delta^i_k: \sF_{i-1} \to d^i_{k *} \sF_i\)
on \(X_{i-1}\) satisfying compatibilities coming from \eqref{eq:ssimp-rel}.
These sheaves form a \emph{topos} \(\tilde{X}_\bullet\) such that morphisms of
semi-simplicial schemes \(f_\bullet: X_\bullet \to Y_\bullet\) induce functorial
maps of topoi \(\tilde{X}_\bullet \to \tilde{Y}_\bullet\) (see \cite[Vbis, Prop.
  1.2.15]{MR0354653}) --- the benefit of
the topos-theoretic point of view is that it immediately implies the category of
\emph{abelian} sheaves \(\Ab(X_\bullet)\) on \(X_\bullet\) is an abelian
category with enough injectives (\cite[\href{https://stacks.math.columbia.edu/tag/01DL}{Tag 01DL}]{stacks-project}), enables us to define pushforward functors
\(Rf_* : D^+(\Ab(X_\bullet)) \to D^+(\Ab(Y_\bullet))\) for morphisms of
semi-simplicial schemes \(f_\bullet: X_\bullet \to Y_\bullet\), and so on.

An \emph{augmented} semi-simplicial scheme is a morphism of semi-simplicial
schemes \(\epsilon_\bullet: X_\bullet \to S_\bullet\) where \(S_\bullet\) is a
\emph{constant} semi-simplicial scheme (that is, \(S_i = S\) for all \(i\) for
some fixed scheme \(S\), and all \(d^i_k = \mathrm{id}\)). This is equivalent to
the data of a semi-simplicial object of \(\Sch_S\). For such a constant
semi-simplicial scheme \(S_\bullet\),
\(\Ab(S_\bullet)  \) is equivalent to the category \(\Ab(S)^{\Lambda}\) of
co-semi-simplicial sheaves of abelian groups on \(S\), that is, sequences of
sheaves of abelian groups \(\sG_i\) on \(S\) together with morphisms
\(\delta_k^i: \sG_{i-1} \to \sG_i\) satisfying compatibilities forced by
\eqref{eq:ssimp-rel}. As in the construction of the \v{C}ech complex setting
\(d^i = \sum_k (-1)^k : \delta_k^i: \sG_{i-1} \to \sG_i\) gives a complex of
abelian sheaves on \(S\) and hence in particular an abelian sheaf
\(a(\sG_\bullet) := \ker d^0\). Writing \(\epsilon_* := a \circ
\epsilon_{\bullet *}\), the composite derived functor
\[
  \begin{tikzcd}
    D^+(\Ab(X_\bullet)) \arrow[rr, bend left, "R\epsilon_*"] \arrow[r, "R\epsilon_{\bullet *}"] & D^+(\Ab(S_\bullet)) \arrow[r, "Ra"] & D^+(\Ab(S))
  \end{tikzcd}
\]
admits the following concrete description: given a sheaf \(\sF_\bullet\) on
\(X_\bullet\) one takes an injective resolution
\[ \sF_\bullet \to \sI_{\bullet}^0 \to \sI_{\bullet}^1 \to \sI_{\bullet}^2 \to \dots
  \text{ in } \Ab(X_\bullet)
\] Here the \(\sI_{\bullet}^j\) are in particular sheaves on \(X_\bullet\) with
each \(\sI_i^j\) an injective abelian sheaf on \(X_i\) --- for further
discussion of injective objects in \(\Ab(X_\bullet)\) see \cite[Vbis, Prop.
  1.3.10]{MR0354653} and \cite[Lem. 6.4, comments on p. 42]{conradCOHOMOLOGICALDESCENT}. Then
\[ \epsilon_{\bullet*} \sI_{\bullet}^0 \to \epsilon_{\bullet*}\sI_{\bullet}^1 \to \epsilon_{\bullet*}\sI_{\bullet}^2 \to \dots
  \text{ in } \Ab(S_\bullet) \] is a complex of co-semi-simplicial abelian sheaves
which via the \v{C}ech construction becomes a complex of complexes. Applying the
sign trick gives a double complex whose \(\Tot\) computes \(R\epsilon_*
\sF_\bullet\). One of the spectral sequences of this
double complex is displayed below. In our calculations it is \emph{crucial} that
this spectral sequence is (at least in a minimal sense) functorial.
\begin{lemma}[{Descent spectral sequence, \cite[Vbis \S 2.3]{MR0354653},
        \cite[Thms. 6.11-6.12]{conradCOHOMOLOGICALDESCENT}}]
  \label{lem:desc-ss}
  If \(\sF_{\bullet}\) is an abelian sheaf on an augmented semi-simplicial scheme
  \(\epsilon: X_\bullet \to S\) then there is a spectral sequence
  \[ E_1^{pq} = R^q \epsilon_{p *} \sF_p \to R^{p+q}\epsilon_* \sF_\bullet \]
  Moreover if \(\sG_{\bullet}\) is an abelian sheaf on another augmented
  semi-simplicial scheme \(\epsilon' : Y_\bullet \to T\) and
  \[
    \begin{tikzcd}
      Y_\bullet \arrow[r, "g_\bullet"] \arrow[d, "\epsilon'"] & X_\bullet \arrow[d, "\epsilon"] \\
      T \arrow[r, "g"] & S
    \end{tikzcd}
  \]
  is a map of augmented semi-simplicial schemes together with a map of abelian
  sheaves \(\varphi: \sF_{\bullet} \to g_{\bullet *}\sG_{\bullet}\) on
  \(X_\bullet\), then \(\varphi\) induces a morphism of spectral sequences
  \[E_1^{pq}(\sF_\bullet) = R^q \epsilon_{p *} \sF_p  \to R^q (\epsilon \circ
    g_p)_* \sG^p = E_1^{pq}(\sG_{\bullet}) \]
  converging to the morphism \(R \epsilon_*(\varphi) :R \epsilon_* \sF_{\bullet} \to R\epsilon_*
  Rg_{\bullet *}\sG_\bullet = Rg_* R \epsilon'_*
  \sG^\bullet  \).
\end{lemma}

\begin{proof}[Proof of the ``Moreover ...'']
  We work with the abelian categories of sheaves of abelian groups on
  \(Y_\bullet\), \(X_\bullet\). Let \(\sJ^\bullet_\bullet\) be an injective
  resolution of \(\sG_\bullet\) in \(\Ab(Y_\bullet)\). Then \(f_{\bullet
      *}\sJ^\bullet_\bullet \) is a complex of injectives (this uses the fact that
  \(f_{\bullet *}\) has an \emph{exact} left adjoint \(f^{-1}\)), \(\sF_\bullet
  \to \sI_{\bullet}^\bullet\) is a quasi-isomorphism and we are given a map
  \[ \varphi: \sF_\bullet \to  f_{\bullet *} \sG_\bullet \to f_{\bullet
        *}\sJ^\bullet_\bullet\]
  By \cite[\href{https://stacks.math.columbia.edu/tag/013P}{Tag
      013P}]{stacks-project} (see also \cite[Thm. 2.2.6]{MR1269324}) there is a map
  of complexes of abelian sheaves on \(X_\bullet\) extending \(\varphi\):
  \[ \tilde{\varphi}: \sI_{\bullet}^\bullet \to  f_{\bullet
        *}\sJ^\bullet_\bullet\]
  Applying \(\epsilon_{\bullet *}\) then gives a morphism of complexes
  of co-semi-simplicial abelian sheaves on \(S\) consisting of morphisms
  \[
    \epsilon_{p *}\sI_p^q \to \epsilon_{p*} g_{p*} \sJ_p^q
  \]
  compatible with \emph{both} the simplicial sheaf maps (in the \(p\) direction)
  \emph{and} the injective resolution maps (in the \(q\)) direction, to which we
  may apply the \v{C}ech construction and sign trick to obtain a map of double
  complexes. This reduces us to the claim that a map of double complexes (or
  more generally a filtered map of filtered complexes) induces a map of spectral
  sequences, which we take as well known.
\end{proof}

\begin{remark}
  The above proof is at least suggested in the last sentence of \cite[Thm.
    6.11]{conradCOHOMOLOGICALDESCENT}. An alternative method would be to use
  Deligne's trick of viewing  \(\varphi\) as an abelian sheaf on the \(\Lambda
  \times I\) scheme associated to \(f_\bullet\) --- for related discussion see
  \cite[Vbis, \S 3.1]{MR0354653}.
\end{remark}

\begin{corollary}
  \label{cor:iso-p1}
  In the situation of \cref{lem:desc-ss} suppose in addition that the morphisms
  \(\varphi_p: \sF_p \to Rf_{p*} \sG_p \) are quasi-isomorphisms for all \(p\).
  Then, the induced morphism
  \[ R \epsilon_*(\varphi) :R \epsilon_* \sF_{\bullet} \to R\epsilon_*
    Rg_{\bullet *}\sG_\bullet = Rg_* R \epsilon'_*
    \sG^\bullet \]
  is a quasi-isomorphism.
\end{corollary}

\begin{corollary}
  \label{cor:bdd-coh}
  In the situation of \cref{lem:desc-ss}, suppose in addition that the scheme
  \(S \) is noetherian, each of the morphisms \(X_i \to S \) is proper, and each
  of the sheaves \(\mathscr{F}_i \) on \(X_i \) is coherent. Then,
  sheaf \(R^i \epsilon_* \sF_{\bullet} \) on \(S \) is coherent for all \(i\)
  and \(R^i \epsilon_* \sF_{\bullet} = 0\) for \(\lvert i \rvert \gg 0\). In
  other words, the complex \(R \epsilon_* \sF_{\bullet}\) belongs to the bounded
  derived category of coherent sheaves \(D_c^b(S)\).
\end{corollary}

\subsection{Regular sequences of divisors}
\begin{definition}
  \label{def:reg-seq-div}
  Let \(X\) be a locally noetherian scheme. A sequence of effective Cartier
  divisors \(D_1, D_2, \dots, D_N \subseteq X\) is called \emph{regular} if and
  only if for each point \(x \in X\), letting \(f_1, \cdots, f_N \in \sO_{X,x}\)
  be local generators for the ideal sheaves \(\mathscr{I}_{D_i}\) at \(x\) and letting \(I(x) = \{i \, | \,
  x \in D_i\}\), the elements \((f_j \in \mathfrak{m}_x \, | \, j \in I(x))\)
  form a regular sequence.
\end{definition}

This definition is designed to ensure that a permutation of a regular sequence
of divisors is again a regular sequence (see \cite[\S 15, Thm. 27]{MR575344}, \cite[\href{https://stacks.math.columbia.edu/tag/00LJ}{Tag 00LJ}]{stacks-project}).

Let \(X \) be a locally noetherian scheme together with a regular sequence of
effective Cartier divisors \(D_1, D_2, \dots, D_N \subseteq X\). We define an
augmented semi-simplicial scheme \(X_\bullet\) as follows: \(X_{-1} = X\), \(X_0
= \coprod_i D_i\) and for \(k>0\),
\[
  X_k =  \coprod_{I \subseteq \{1,\dots, N\} \, | \, |I| = k+1} D_I, \text{  where  } D_I = \bigcap_{j \in I} D_j
\]
The face maps are defined by the inclusions \(d_k^j : D_I \inj D_{I
    \setminus \{i_j\}}\) for \(I = \{i_0, \dots, i_k\}\) and \(0 \leq j \leq i\), as
in a \v{C}ech complex, and for each \(k\) we have an augmentation map
\(\epsilon_p : X_k \to X\) obtained from the inclusions \(D_I \subseteq X\).
In this situation the descent spectral sequence of \cref{lem:desc-ss}
degenerates: since the \(\epsilon_p: X_p \to X\) are closed immersions and hence
affine, \(R^q \epsilon_{p*} \sO_{X_p} = 0 \) for \(q>0\). It follows that \(R^i
\epsilon_* \sO_{X_\bullet}\) is the cohomology of the \v{C}ech type complex
\begin{equation}
  \label{eq:cech-type-cx}
  \begin{split}
    & \epsilon_{0 *} \sO_{X_0} \xrightarrow{d^1}  \epsilon_{1 *} \sO_{X_1} \xrightarrow{d^2} \cdots \xrightarrow{d^N} \epsilon_{N *} \sO_{X_N}  \\
    &= \bigoplus_i \sO_{D_i} \xrightarrow{d^1} \bigoplus_{i < j} \sO_{D_i \cap D_j} \xrightarrow{d^2} \cdots \xrightarrow{d^N}\sO_{\cap_i D_i}
  \end{split}
\end{equation}
\begin{lemma}
  \label{lem:cech-cx-reg-seq-Cdiv}
  The complex \eqref{eq:cech-type-cx} is exact in degrees \(i>0\), with \(\ker
  d^1 \simeq \sO_{\cup_i D_i}\). Equivalently, the extended complex
  \[
    0 \to \sO_X(-\sum_i D_i) \to \sO_X \xrightarrow{\gamma} \bigoplus_i \sO_{D_i} \xrightarrow{d^1} \bigoplus_{i < j} \sO_{D_i \cap D_j} \xrightarrow{d^2} \cdots \xrightarrow{d^N}\sO_{\cap_i D_i} \to 0
  \] where \(\gamma:  \sO_X \to \bigoplus_i \sO_{D_i}\) is restriction in each factor is exact, and hence there is a canonical quasi-isomorphism \(\sO_X(-\sum_i D_i) \simeq \cone ( \sO_X \to R\epsilon_* \sO_{X_\bullet} )[-1]\).
\end{lemma}

\begin{proof}
  We proceed by induction on the number \(N\) of divisors. The base case \(N=0\)
  is vacuous (\(X_\bullet\) is empty). If that seems too weird, the case \(N=1\)
  simply says that the sequence \(0\to \sO_X(-D_1) \to \sO_X \to \sO_{D_1} \to
  0\) is exact, which is indeed the case as \(D_1\) is an effective Cartier divisor.

  Suppose now that \(N>1\). Then by the definition of a regular sequence, \(D_1
  \cap D_2, D_1 \cap D_3, \dots, D_1 \cap D_N \subseteq D_1 \) is a regular
  sequence of divisors, and by permutation invariance of regular sequences (for
  \emph{noetherian local rings} \cite[\S 15, Thm. 27]{MR575344}, \cite[\href{https://stacks.math.columbia.edu/tag/00LJ}{Tag 00LJ}]{stacks-project} --- this dictated \cref{def:reg-seq-div}) \(D_2,
  \dots, D_N \subseteq X\) is a regular sequence. We form a short exact sequence
  of complexes (with cohomological degrees as indicated)
  \begin{equation}
    \label{eq:reg-seq-induction}
    \begin{tikzcd}[column sep=small]
      C' \arrow[d, "\alpha"]: & 0 \arrow[r] \arrow[d] & \sO_{D_1} \arrow[r, "d'"]\arrow[d, "\alpha"] & \bigoplus_{1<j} \sO_{D_1 \cap D_j}  \arrow[r, "d'"] \arrow[d, "\alpha"] & \bigoplus_{1<j<k}\sO_{D_1 \cap D_j\cap D_k} \arrow[r, "d'"] \arrow[d, "\alpha"] & \cdots \\
      C \arrow[d, "\beta"]: & \sO_X \arrow[r, "\gamma"]\arrow[d, equals, "\beta"] & \bigoplus_{i} \sO_{D_i} \arrow[r, "d"] \arrow[d, "\beta"] & \bigoplus_{i<j} \sO_{D_i\cap D_j} \arrow[r, "d"] \arrow[d, "\beta"] & \bigoplus_{i<j<k} \sO_{D_i\cap D_j\cap D_k}  \arrow[r, "d"] \arrow[d, "\beta"] & \cdots \\
      C'': &\sO_X \arrow[r, "\gamma''"] &   \bigoplus_{1<i} \sO_{D_i} \arrow[r, "d''"] & \bigoplus_{1<i<j} \sO_{D_i\cap D_j} \arrow[r, "d''"] &  \bigoplus_{1<i<j<k} \sO_{D_i\cap D_j\cap D_k}  \arrow[r, "d''"] & \cdots \\
      & -1 & 0 & 1 & 2 &
    \end{tikzcd}
  \end{equation}
  (in fact by comparing ranges of indices we can see the columns are \emph{split} short exact sequences). By inductive hypotheses,
  \[ h^i(C') = \begin{cases}
      \sO_{D_1}(-\sum_{1<j}D_1 \cap D_j) & \text{ if } i = 0  \\
      0                                  & \text{  otherwise}
    \end{cases} \text{ and }
    h^i (C'') = \begin{cases}
      \sO_{X}(-\sum_{1<j}D_j) & \text{ if } i = -1 \\
      0                       & \text{  otherwise}
    \end{cases}
  \] showing that \(h^i(C) = 0 \) for \(i > 0\), and that in low degrees there
  is an exact sequence
  \begin{equation}
    \label{eq:disguised-restriction}
    0 \to h^{-1} (C) \to \sO_{X}(-\sum_{1<j}D_j) =  h^{-1} (C'')
    \xrightarrow{\delta} h^0(C') = \sO_{D_1}(-\sum_{1<j}D_1 \cap D_j) \to 0
  \end{equation}
  To complete the proof, we must verify that the connecting map \(\delta\) is
  indeed restriction of sections, so that \eqref{eq:disguised-restriction}
  coincides with the usual exact sequence
  \[ 0 \to \sO_X(-\sum_{j}D_j) \to \sO_{X}(-\sum_{1<j}D_j) \to
    \sO_{D_1}(-\sum_{1<j}D_1 \cap D_j) \to 0 \]
  and indeed, by the snake lemma construction of the connecting map \(\delta\)
  we lift a local section  \(\sigma \in \ker \gamma'' \subseteq \sO_X\) along \(\beta\), apply \(\gamma:
  \sO_X \to \oplus_i \sO_{D_i}\) to obtain a local section \( (\sigma|_{D_i})
  \in \ker \beta \subseteq \oplus_i \sO_{D_i} \), and then lift along \(\alpha:
  \sO_{D_1} \to \oplus_i \sO_{D_i}\) --- the net result is \(\sigma|_{D_1}\) as claimed.
\end{proof}

\begin{remark}
  Here we sketch a different proof of \cref{lem:cech-cx-reg-seq-Cdiv}, which
  could potentially shed more light on what happens if \(D_1, \dots, D_N
  \subseteq X\) deviates from being a regular sequence. For each \(i\) let
  \(\sigma_i : \sO_X \to \sO_X(D_i  )\) be the canonical global section and let
  \(\sigma_i^\vee : \sO_X(-D_i) \to \sO_X\) be its dual. For each subset \(J
  \subseteq \{1, \dots, N\}\) let \(\sE_J := \oplus_{j \in J} \sO_X(D_j)\). For
  each such \(J\) we have a section \(\sigma_J = (\sigma_j | j \in J): \sO_X \to
  \sE_J\).  There's a map of chain complexes
  \[
    \begin{tikzcd}
      0 = \sE_{\emptyset} \arrow[r] & \bigoplus_{|J| =1} \sE_J \arrow[r] & \bigoplus_{|J| = 2} \sE_J \arrow[r] &\bigoplus_{|J| =3} \sE_J \arrow[r] &\cdots\\
      \sO_X \arrow[r] \arrow[u] & \bigoplus_{|J| = 1} \sO_X \arrow[r] \arrow[u, "\oplus \sigma_J"] & \bigoplus_{|J| = 2}  \sO_X \arrow[r] \arrow[u, "\oplus \sigma_J"] & \bigoplus_{|J| = 3}  \sO_X \arrow[u, "\oplus \sigma_J"] \arrow[r] & \cdots
    \end{tikzcd}
  \] where the horizontal differentials are alternating sums of summand inclusions (in effect, they come from the singular co-chain complex of the \(N-1\)-simplex \(\Delta^{N-1}\)) and the vertical maps are induced by the \(\sigma_i\). Applying the Koszul construction to the individual maps \( \sigma_J :\sO_X \to \sE_J\) (along with the usual sign trick) then results in a double complex \(C^{\bullet \bullet}\) with \(C^{pq} = \bigoplus_{|J| = p} \wedge^{-q} \sE_J^\vee \).

  I conjecture\footnote{It seems a proof by induction on \(N\)
    analogous to the argument in \cref{lem:cech-cx-reg-seq-Cdiv} works, although
    it is combinatorially more involved.} that the \emph{horizontal} complexes
  \[ C^{\bullet q} : 0 \to \cdots \to 0 \to \bigoplus_{|J| = -q} \wedge^{-q}
    \sE_J^\vee \to \bigoplus_{|J| = -q+ 1} \wedge^{-q} \sE_J^\vee \to \cdots \to
    \bigoplus_{|J| = N} \wedge^{-q} \sE_J^\vee  =  \wedge^{-q} (\bigoplus_{i=1}^N \sO_X(-D_i)) \] are
  exact for \(q > -N\), and hence \(\Tot (C^{\bullet\bullet})\) is
  quasi-isomorphic to \( \wedge^{N} (\oplus_{i=1}^N \sO_X(-D_i)) =\sO_X(-\sum_i
  D_i)\). On the other hand, the vertical complexes
  \[C^{p \bullet } : 0 \to \cdots \to 0 \to \bigoplus_{|J| = p} \wedge^{p}
    \sE_J^\vee \to \bigoplus_{|J| = p} \wedge^{p-1} \sE_J^\vee \to \cdots \to
    \bigoplus_{|J| = p} \sE_J^\vee \to \bigoplus_{|J| = p} \sO_X\] are direct sums
  of Koszul complexes by design, and so their cohomology is
  \[h^q(C^{p \bullet}) = \bigoplus_{|J| = p} \sTor_{-q}^{\sO_X}(\sO_{D_J},
    \sO_X),\]
  which reduces to
  \[
    h^q(C^{p \bullet}) =
    \begin{cases}
      \bigoplus_{|J| = p} \sO_{D_J} & \text{ if } q=0  \\
      0                             & \text{otherwise}
    \end{cases}
  \] \emph{precisely} when the sequence \(D_1,\dots, D_N\) is regular \cite[\S 18 Thm. 43]{MR575344}, \cite[Lem. A.5.3]{MR1644323}. As a technical aside, this approach might show that \cref{lem:cech-cx-reg-seq-Cdiv} holds under slightly weaker hypotheses of \emph{Koszul regularity} (see e.g. \cite[\href{https://stacks.math.columbia.edu/tag/062D}{Tag 062D}]{stacks-project}).
\end{remark}

\subsection{Replacing the ideal sheaf with a filtered complex}
\label{sec:replacing-O-delta}
Let \(X\) be a locally noetherian scheme and let \(D_1, \dots, D_N \subseteq X\)
be a regular sequence of effective Cartier divisors, with sum \(\Delta_X :=
\sum_{i=1}^N D_i\). By  \cref{lem:cech-cx-reg-seq-Cdiv} the ideal sheaf
\(\lstrshf{X}\) is quasi-isomorphic to \(\cone (\sO_X \to R\epsilon_*
\sO_{X_\bullet} )[-1]\), which for convenience moving forward we give a
name:\footnote{This notation is chosen to align with the fact that over \(\CC\)
  and when \(\lsp{X}\) is a simple normal crossing pair, the complex
  \eqref{eq:res-of-log-hodge} represents the \(0\)th graded part of the Du Bois
  complex of the pair \(\lsp{X}\).}
\begin{definition}
  \(\underline{\Omega}_{X, \Delta_X}^0 := \cone (\sO_X \to R\epsilon_* \sO_{X_\bullet} )[-1]\).
\end{definition}
By \cref{lem:cech-cx-reg-seq-Cdiv} and its proof this complex has the explicit
representation
\[
  \begin{tikzcd}[row sep=small]
    \sO_X \arrow[r] & \bigoplus_i \sO_{D_i} \arrow[r] & \bigoplus_{i < j} \sO_{D_i \cap D_j} \arrow[r] & \cdots \arrow[r] & \sO_{\cap_i D_i} \\
    0 & 1 & 2 & \dots & N
  \end{tikzcd}
\]
We can give \(\underline{\Omega}_{X, \Delta_X}^0\) a descending filtration by truncations
\[\underline{\Omega}_{X,\Delta_X}^0 = \sigma_{\geq
    0}\underline{\Omega}_{X,\Delta_X}^0 \supset \sigma_{\geq
    1}\underline{\Omega}_{X,\Delta_X}^0 \supset \sigma_{\geq
    2}\underline{\Omega}_{X,\Delta_X}^0 \supset \cdots \] where
\begin{equation}
  \label{eq:trunc-filt}
  (\sigma_{\geq i}
  \underline{\Omega}_{X,\Delta_X}^0)^j =
  \begin{cases}
    0                                                              & \text{ if } j <i    \\
    (\underline{\Omega}_{X,\Delta_X}^0)^j = \epsilon_{j-1*} \sO_{X_{j-1}} =
    \prod_{J \subseteq \{1,\dots, N\} \, | \, |J| = j\}} \sO_{D_J} & \text{  otherwise }
  \end{cases}
\end{equation}
Using this filtration we obtain a spectral sequence for higher
direct images.

\begin{corollary}
  \label{cor:desc-ss-hdi}
  Let \(S\) be a locally noetherian scheme and let \(f: X \to S\) be a finite
  type morphism. Let \(D_1, \dots, D_N\subseteq X\) be a regular sequence of
  effective Cartier divisors, with sum \(\Delta_X\). Then there is a filtered
  complex \((Rf_* \underline{\Omega}_{X,\Delta_X}^0, F)\) whose cohomology
  computes the higher direct images \(R^{i+j}f_* \lstrshf{X}\). For each \(i\)
  there is a distinguished triangle
  \[ F^{i+1} Rf_* \underline{\Omega}_{X,\Delta_X}^0 \to  F^{i} Rf_*
    \underline{\Omega}_{X,\Delta_X}^0 \to \prod_{J \subseteq \{1,\dots, N\} \, |
      \, |J| = i\}} Rf_* \sO_{D_J} \to \cdots
  \]
  In particular, there is a spectral sequence
  \[ E_1^{ij} = \prod_{J \subseteq \{1,\dots, N\} \, |
      \, |J| = i\}} R^jf_* \sO_{D_J}\implies R^{i+j}f_* \lstrshf{X} \]
\end{corollary}

The filtration \(F\) is defined as \(F = Rf_* \sigma\), and the resulting
spectral sequence is just the usual hypercohomology spectral sequence.

\begin{remark}
  Viewing \(\epsilon: X_\bullet \to X\) as a sort of resolution of the pair
  \(\lsp{X}\), we can consider the spectral sequence of \cref{cor:desc-ss-hdi}
  as a sort of \emph{descent} spectral sequence (see \cite[Vbis]{MR0354653},
  \cite{conradCOHOMOLOGICALDESCENT}).
\end{remark}

\section{Simple normal crossing divisors and thriftyness}
\label{sec:snc-thrifty}
\subsection{Definitions and basic properties}

\begin{definition}[{\cite[\S 7.8]{MR199181}}]
  \label{def:exc}
  A scheme \(X\) is \emph{excellent} if and only if
  \begin{itemize}
    \item \(X\) is locally noetherian,
    \item for every point \(x \in X\) the fibers of the natural map \(\Spec
          \sO_{X, x}^{\wedge} \to \Spec \sO_{X, x}\) are regular,
    \item for every integral \(X\)-scheme \(Z\) that is finite over an affine
          open of \(X\), there is a non-empty regular open subscheme \(U \subseteq
          Z\), and
    \item every scheme \(X'\) locally of finite type over \(X\) is catenary
          (that is, if \(x \in X'\) and  \(x \rightsquigarrow y\) is a specialization,
          then any 2 saturated chains of specializations \(x = x_0 \rightsquigarrow
          x_1 \rightsquigarrow \cdots \rightsquigarrow x_n = y\) have the same length).
  \end{itemize}
\end{definition}
If \(X\) is excellent, then the locus
\[\Reg(X) = \{x \in X \, | \, \sO_{X, x} \text{  is regular}\}\]
is open \cite[Prop. 7.8.6]{MR199181}; we will make repeated use of this fact.

We first relate the notion of a simple normal crossing pair to the regular
sequences of effective Cartier divisors considered in the previous section.

\begin{lemma}
  \label{lem:tcp-reg-seq}
  If \((X, \Delta_X = \sum_i D_i)\) is a simple normal crossing pair then
  \((D_i)\) is a regular sequence of effective Cartier divisors.
\end{lemma}

\begin{proof}
  Let \(x \in X\) be a point and as above let \(I(x) = \{i \, | \, x  \in
  D_i\}\). Let \(f_j \in \mathfrak{m}_x \subseteq \sO_{X,x}\) be local
  generators for the \(D_j\), for \(j \in I(x)\). By hypothesis for any subset
  \(J \subseteq I(x)\) the quotient \(A/(f_j \, | \, j \in J)\) is regular, and
  so by induction we reduce to the commutative algebra statement that if \(A\)
  is a regular local ring, \(f \in A\) and \(A/f\) is a regular local ring with
  dimension \(\dim A -1\) then \(f \) is a non-0-divisor (see for example \cite[\href{https://stacks.math.columbia.edu/tag/0AGA}{Tag 0AGA}]{stacks-project}).
\end{proof}

\begin{lemma}
  \label{lem:tc-locus-open2}
  Let \(X\) be an integral excellent scheme with an effective Weil
  divisor  \(\Delta_X = \sum_i D_i\), and for each \(i \) let \(\sI_i
  \subseteq \sO_X\) be the ideal sheaf of \(D_i\). Then the locus
  \[ \snc(X, \Delta_X) := \{ x \in X \, | \, \sum_{ i \in I(x)} \sI_{i, x}^\wedge
    \subseteq \sO_{X,x}^{\wedge} \text{ is a simple normal crossing pair} \}
    \subseteq X \] is open, and this is the largest open set \(U \subseteq X\)
  such that \((U, \Delta_X|_U)\) is a simple normal crossing pair.
\end{lemma}

We could alternatively just declare \(\snc(X, \Delta_X)\) to be the largest open
set \(U \subseteq X\) such that \((U, \Delta_X|_U)\) is a simple normal crossing
pair; the content of the lemma is that in some sense the snc locus is ``already open.''

\begin{proof}
  Suppose \(J \subseteq \{1,\dots, N\}\), and write \(\sI_J = (f_j \in
  \sO_{X,x}\, | \, j \in J) \subseteq \sO_{X,x}\). Consider the co-cartesian
  diagram of noetherian local rings
  \[
    \begin{tikzcd}
      \sO_{X, x}^\wedge \arrow[r] & \sO_{X, x}^\wedge/\sI_J \sO_{X, x}^\wedge \simeq (\sO_{X,x}/\sI_J)^\wedge \\
      \sO_{X, x} \arrow[r] \arrow[u] & \sO_{X,x}/\sI_J \arrow[u]
    \end{tikzcd}
  \]
  The vertical homomorphisms are faithfully flat and by hypothesis \(\sO_{X,
    x}^\wedge/\sI_J \sO_{X, x}^\wedge\) is regular --- since regularity satisfies
  faithfully flat descent, \(\sO_{X,x}/\sI_J\) is also regular. Thus \(D_J\) is
  regular at the point \(x \in D_J\), and as \(X\) is excellent by hypothesis
  the regular locus of \(D_J\) is open. Letting \(x \in U_J \subseteq X\) be a
  neighborhood such that \(D_J \cap U_J \subseteq D_J\) is regular and then
  letting \(U = \cap_{J} D_J\) gives a neighborhood of \(x\) such that \((U,
  (D_i \cap U))\) is a simple normal crossing pair.
\end{proof}

Note that for a simple normal crossing pair \(\lsp{X}\), since the intersections
\(D_J = \cap_{j \in J}D_j\) are regular their connected components and
irreducible components coincide.  For convenience we recall the definitions of
strata and thriftiness mentioned in the Introduction.
\begin{definition}
  \label{def:abstract-strata}
  A \emph{stratum} of a simple normal crossing pair \((X, \Delta_X = \sum_i D_i)\) is a connected
  (equivalently, irreducible) component of an intersection \(D_J = \cap_{j \in J}D_j\).
\end{definition}

\begin{definition}[{compare with \cite[Def. 2.79-2.80]{MR3057950}, \cite[\S 1,
          discussion before Def. 10]{MR3539921}}]
  \label{def:abstract-thriftiness}
  Let \((S, \Delta_S = \sum_i D_i)\) be a pair in the sense of \cref{def:pair},
  and assume
  \(\Delta_S\) is reduced and effective. A
  separated, finite type birational morphism \(f: X \to S\) is \emph{thrifty
    with respect to} \(\Delta_S \) if and only if
  \begin{enumerate}
    \item \label{item:over-gen-pts} \(f\) is an isomorphism over the generic point of every stratum of
          \(\snc(S, \Delta_S)\) and
    \item \label{item:at-gen-pts} letting \(\tilde{D}_i = f^{-1}_*D_i\) for \(i
          = 1, \dots, N\) be the strict transforms of the divisors \(D_i\), and
          setting \(\Delta_X := \sum_i \tilde{D}_i\), the map \(f\) is
          an isomorphism at the generic point of every stratum of \(\snc(X,
          \Delta_X)\).
  \end{enumerate}
\end{definition}

The restriction that \(D_i \cap \Reg(S) \neq \emptyset\) for all \(i\) ensures
that if \(\eta \in D_i\) is a generic point of a component, then \(\eta \in
\Reg(S)\). Since on a regular scheme every Weil divisor is Cartier, and as \(S\)
is excellent and \(D_i\) is reduced by hypothesis, there is a neighborhood
\(\eta \in U \subseteq S\) such that \(U, D_i\cap U\) is a simple normal crossing
pair. In other words, \(\eta \in \snc(S, \Delta_S)\) is the generic point of a
stratum, so \cref{item:over-gen-pts} implies  \(f^{-1}(\eta)\) is a single
(non-closed) point. For our purposes the strict transform \(\tilde{D}_i\) can be
\emph{defined} as
\[\tilde{D}_i := \bigcup_{\eta \in D_i \text{generic}} \overline{f^{-1}(\eta)}
  \subseteq X.\] Since \(f\) is an isomorphism over \(\eta\), we also see
\(f^{-1}(\eta) \subseteq \snc(X, \Delta_X)\).

\begin{lemma}
  \label{lem:common-open}
  Let \(S\) be an integral excellent noetherian scheme with a sequence of
  reduced effective Weil divisors \(D_1,\cdots, D_N \subseteq S\) such that no
  component of \(\cup_i D_i\) is contained in \(\Sing(X)\), and let \(f: X \to
  S\) be a separated, finite type birational morphism. Then, \(f\) is thrifty
  if and only if there is a diagram of separated finite type \(S\)-schemes
  \[ S \hookleftarrow U \hookrightarrow X \] with both morphisms (necessarily
  dense) open immersions, such that \(U\) contains all generic points of strata
  of \(\snc(S, \Delta_S) \) and \(\snc(X,\Delta_X)\).
\end{lemma}

\begin{proof}
  Since the existence of a common dense open \( S \hookleftarrow U
  \hookrightarrow X \) as in the statement of the lemma certainly guarantees
  \cref{item:at-gen-pts,item:over-gen-pts}, we focus on the ``only if,'' and in
  fact we show that one can take \(U =\) the maximal domain of definition of
  \(f^{-1}: S \dashrightarrow X\). By \cref{item:over-gen-pts} of
  \cref{def:abstract-thriftiness} this \(U \) contains all generic points of
  strata of \(\snc (S, \Delta_S)\).

  Suppose \(\xi \in \snc(X,\Delta_X)\) is a generic point of a stratum. By
  hypothesis there is a neighborhood \(\xi \in V \subseteq X\) such that \(f|_V:
  V \isom S\) is an isomorphism onto its image. Then \(W :=f(V)\) is a Zariski
  neighborhood of \(f(\xi)\) and the inverse of \(f|_V\) gives a section of the
  birational map \(X_W = X\times_S W \to W\).
  \[
    \begin{tikzcd}
      V \arrow[dr, "f|_V"'] \arrow[r, hook] & X_W \arrow[d, "f_W"] \\
      & W  \\
    \end{tikzcd}
  \]
  But then the inclusion \(V \hookrightarrow X_W\) is a proper dense open
  immersion, hence an isomorphism.
\end{proof}

\begin{remark}
  It seems that the above proof shows in addition that \(f(\xi) \in S\) is the
  generic point of a stratum of \(\snc(S, \Delta_S)\).
\end{remark}

We will make repeated use of a few blowup lemmas from the construction of Nagata
compactifications in \Cref{sec:constructions} --- here, they are used to show
that thrifty morphisms can be dominated by certain admissible blowups.
\begin{lemma}[{\cite[Lem. 2.4, Rmk. 2.5, Cor. 2.10]{MR2356346}}]
  \label{lem:common-blp}
  Let \(S\) be a quasi-compact, quasi-separated scheme.

  \begin{enumerate}
    \item If \(X\) is a quasi-separated quasi-compact \(S\)-scheme and \(Y\) is
          a proper \(S\)-scheme, and if \(f: U \to Y\) is an \(S\)-morphism defined on
          a dense open \(U \subseteq X\), then there exists a \(U\)-admissible blowup
          \(\tilde{X} \to X\) and an \(S\)-morphism \(\tilde{f}: \tilde{X} \to Y\)
          extending \(f\).

    \item Let \(j_i: U \to X_i\) be a finite collection of dense open immersions
          between finite type separated \(S\)-schemes. Then there exist
          \(U\)-admissible blowups \(X'_i \to X_i\) and a separated finite type
          \(S\)-scheme \(X\), together with open immersions \(X'_i  \inj X\) over
          \(S\), such that the \(X'_i\) cover \(X\) and the open immersions \(U \inj
          X'_i \inj X\) are all the same.
  \end{enumerate}
\end{lemma}

\begin{corollary}
  \label{cor:common-adm-blp}
  There exist \(U\)-admissible blowups
  \[
    \begin{tikzcd}[column sep=large]
      \tilde{X} \arrow[r, hook, "\text{open imm.}"] \arrow[d] & \tilde{S} \arrow[d] \\
      X \arrow[r, "f"] & S
    \end{tikzcd}
  \] In particular if \(f\) is proper then \(X\) and \(S\) have a common \(U\)-admissible blowup.
\end{corollary}

\begin{proof}
  By \cref{lem:common-blp} there are a separated, finite type \(S\)-scheme
  \(Y\), \(U\)-admissible blowups \(\tilde{S}\to S\) and \(\tilde{X} \to X\) and
  dense open immersions \(\tilde{S} \hookrightarrow Y \hookleftarrow \tilde{S}\)
  over \(S\) such that the diagram
  \[
    \begin{tikzcd}
      U \arrow[r, hook] \arrow[d, hook] & \tilde{X} \arrow[d,hook] \\
      \tilde{S} \arrow[r, hook] & Y
    \end{tikzcd}
  \]
  commutes. Since \(\tilde{S}\) is proper over \(S\), the bottom arrow is
  necessarily an isomorphism, in other words \(Y= \tilde{S}\). If \(f\) is
  proper then \(\tilde{X}\) is proper over \(S\), so \(Y = \tilde{X}\) as well.
\end{proof}

\begin{remark}
  If \(\lsp{S}\) is a simple normal crossing pair and \(U \subseteq S\) is
  an open containing all strata, a \(U\)-admissible blowup \(f: X \to S\) need
  not be thrifty, see \cref{ex:adm-not-thrifty}.
\end{remark}

\subsection{The regular-to-regular case}

Using \cref{cor:desc-ss-hdi} we can already obtain a restricted form of
\cref{thm:hdi-log-struct}, the case of a thrifty proper birational morphism of
simple normal crossing pairs. In the proof we will make use of Grothendieck
duality, as formulated in \cite{MR0222093,MR1804902}.

\begin{theorem}[{Grothendieck duality, \cite[Cor. VII.3.4]{MR0222093},
        \cite[Thm. 3.4.4]{MR1804902}}]
  \label{thm:GD}
  Let \( f : X \to Y \) be a proper morphism of finite-dimensional noetherian
  schemes and assume \(Y\) admits a dualizing complex (for example \(X \) and
  \(Y \) could be schemes of finite type over \(k \)). Then for any pair of objects \(
  \mathscr{F}^\bullet \in D_{qc}^-(X)  \) and \(\sG^\bullet \in D_c^+(Y)\)
  there is a natural isomorphism
  \[ Rf_* R \underline{Hom}_X(\mathscr{F}^\bullet, f^{!}\sG^\bullet) \simeq
    R\underline{Hom}_Y(Rf_* \mathscr{F}^\bullet, \sG^\bullet) \text{ in }
    D_c^b(Y) \]
\end{theorem}
If \(\omega_Y^\bullet\) is a dualizing complex on \(Y\) then \(\omega_X^\bullet
:= f^! \omega_Y^\bullet\) is a dualizing complex on \(X\) \cite[\S V.10, Cor.
  VI.3.5]{MR0222093}, and so in the case  \(\sG = \omega_Y^\bullet\) we obtain a natural isomorphism
\[ Rf_* R \underline{Hom}_X(\mathscr{F}^\bullet, \omega_X^\bullet) \simeq
  R\underline{Hom}_Y(Rf_* \mathscr{F}^\bullet, \omega_Y^\bullet) \text{ in }
  D_c^b(Y) \] When \(X \) is a Cohen-Macaulay \(n\)-dimensional
noetherian scheme (so in particular when \(X\) is a regular \(n\)-dimensional
noetherian scheme), a dualizing complex \(\omega_X^\bullet\) satisfies \(h^i
\omega_X^\bullet  = 0\) for \(i \neq -n\), and the unique non-0 cohomology sheaf
is called the dualizing \emph{sheaf} and denoted \(\omega_X := h^{-n} \omega_X^\bullet \).

\begin{theorem}
  \label{thm:smth-to-smth}
  Let \((Y, \Delta_Y)\) be a simple normal crossing pair and let \(f: X \to Y
  \) be a thrifty proper birational morphism. Assume \((X, \Delta_X)\) is
  also a simple normal crossing pair. Then, the natural map
  \[ \lstrshf{Y} \to Rf_* \lstrshf{X} \text{  is a  quasi-isomorphism,}\]
  and it follows that \(Rf_* \omega_X(\Delta_X) \simeq \omega_Y(\Delta_Y) \).
\end{theorem}

\begin{proof}
  Let \(X_\bullet\) (resp. \(Y_\bullet\)) be the semi-simplicial scheme
  associated to \(\lsp{X}\) (resp. \(\lsp{Y}\)). For any \(J \subseteq
  \{1,\dots, N\}\) \(f\) restricts to a morphism \( \cap_{j \in J} \tilde{D}_j
  \to \cap_{j \in J} D_j\), and in this way we obtain a morphism of
  semi-simplicial schemes
  \begin{equation}
    \label{eq:ind-mp-ssimpsch}
    \begin{tikzcd}
      \cdots \arrow[r] & X_2 \arrow[r] \arrow[d, "f_2"] & X_1 \arrow[r] \arrow[d, "f_1"] & X_0 \arrow[r] \arrow[d, "f_2"] &X \arrow[d, "f"] \\
      \cdots \arrow[r] & Y_2 \arrow[r] & Y_1 \arrow[r] & Y_0 \arrow[r] & Y
    \end{tikzcd}
  \end{equation}
  The hypothesis that both pairs have simple normal crossings and \(f\) is thrifty
  implies that for each \(i\), \(f_i : X_i \to Y_i\) is a proper birational
  morphism of (possibly disconnected) regular schemes over \(k\). By \cite[Cor.
    3.2.10]{MR2923726} (or \cite[Thm. 1.1]{MR3427575}, \cite[Thm. 1.4]{kovacsRationalSingularities2022})
  \begin{equation}
    \label{eq:cr-on-strata}
    \strshf{Y_i} \isom Rf_* \strshf{X_i} \text{ is a quasi-isomorphism for all } i
  \end{equation}
  The diagram \eqref{eq:ind-mp-ssimpsch} induces a morphism of \emph{filtered}
  complexes \(f^\sharp : \underline{\Omega}_{Y, \Delta_Y}^0 \to
  Rf_*\underline{\Omega}_{X, \Delta_X}^0 \), and by \cref{lem:cech-cx-reg-seq-Cdiv} and
  \cref{cor:desc-ss-hdi} it will suffice to show that the resulting map of
  descent spectral sequences
  \[E_1^{ij}(Y) =
    \begin{cases}
      \prod_{\sigma \in
      \dcx{\Delta_Y}^{i-1}}\strshf{D(\sigma)} & j=0                \\
      0                                       & \text{  otherwise}
    \end{cases}
    \to \prod_{\sigma \in \dcx{\Delta_X}^{i-1}}R^j f_* \strshf{D(\sigma)} =
    E_1^{ij}(X) \]
  is an isomorphism, and this last step is a consequence of
  \eqref{eq:cr-on-strata}.

  Finally, once we know \( \lstrshf{Y} \simeq Rf_* \lstrshf{X}\), applying the
  functor \(R\sHom(-, \omega_Y^\bullet)\) we see that
  \begin{equation}
    R\sHom(\lstrshf{Y}, \omega_Y^\bullet) \simeq R\sHom(Rf_* \lstrshf{X}, \omega_Y^\bullet) \simeq Rf_* R\sHom(\lstrshf{X}, \omega_X^\bullet)
  \end{equation}
  where the second quasi-isomorphism comes from \cref{thm:GD}. Since \(\lsp{X}\)
  and \(\lsp{Y}\) are simple normal crossing pairs by hypothesis, in particular
  we know that the divisors \(\Delta_X \) and \(\Delta_Y \) are Cartier, so that
  \begin{equation}
    R\sHom(\lstrshf{Y}, \omega_Y^\bullet) \simeq \omega_Y^\bullet(\Delta_Y) \text{  and  } Rf_* R\sHom(\lstrshf{X}, \omega_X^\bullet) \simeq Rf_* \omega_X^\bullet(\Delta_X).
  \end{equation}
\end{proof}

\section{Constructing semi-simplicial projective Cohen-Macaulayfications}
\label{sec:constructions}
\subsection{Preliminaries}
In the situation of \cref{thm:hdi-log-struct}, if \(Z\) is smooth and
\(\Delta_Z\) is snc, then \cref{thm:smth-to-smth} applied to both \(r\) and
\(s\) shows \[Rf_* \lstrshf{X} \simeq Rf_* Rr_* \lstrshf{Z} =  Rg_* Rs_*
  \lstrshf{Z} \simeq Rg_* \lstrshf{Y}. \] Of course, \(Z\) need not be smooth and
in the absence of resolution of singularities away from characteristic
0,\footnote{At least at the time of this writing ...} we cannot replace it by a
resolution. In characteristic \(p>0\) we could replace \(Z\) with an
alteration, but only at the cost of allowing \(r, s\) to be generically finite
but not necessarily birational, and as such using alterations seems incompatible
with the strategy of \cref{thm:smth-to-smth}. Moreover, to the best of our
knowledge at the level of generality \cref{thm:hdi-log-struct} is stated, even alterations are unavailable.\footnote{Ditto.}

Instead, we will replace \(Z\) with a mildly singular (specifically
Cohen-Macaulay and normal)
semi-simplicial scheme \(Z_\bullet\) together with morphisms \(X_\bullet
\xleftarrow{r_\bullet} Z_\bullet \xrightarrow{s_\bullet} Y_\bullet\) over \(S\)
which are term-by-term proper birational equivalences over \(S\). It is in this
construction that we need \cref{conj:macify}, restated here for convenience:

\begin{conjecture}[{see also \cite[Conj. 1.1]{CesMac}, \cite[Conj.
          1.14]{kovacsRationalSingularities2022}}]
  \label{conj:macify-recall}
  For every CM-quasi-excellent noetherian scheme \(X\) there exists a
  projective birational morphism \(\pi : \tilde{X} \to X\) such that
  \(\tilde{X}\) is Cohen-Macaulay and normal and \(\pi\) is an isomorphism over the
  regular locus \(\mathrm{Reg}(X) \subset X\).
\end{conjecture}
The usefulness of normal Cohen-Macaulayfications for the problem at hand
stems from an extension of the results of Chatzistamatiou-R\"ulling due to
Kov\'acs.
\begin{theorem}[{\cite[Thm. 1.4]{kovacsRationalSingularities2022}}]
  \label{thm:sk-ratlsings}
  Let \(f: X \to Y\) be a locally projective birational morphism of excellent
  Cohen-Macaulay normal schemes. If \(Y\) has pseudo-rational singularities then
  \[\strshf{Y} = Rf_* \strshf{X} \text{  and  } Rf_* \omega_X = \omega_Y.\]
\end{theorem}
By a result of Lipman-Teissier, if \( Y\) is regular (so in particular if it is
smooth over \(k\)) then \(Y\) is pseudo-rational \cite[\S 4]{MR600418}, hence
\cref{thm:sk-ratlsings} applies when \(Y\) is regular.

\subsection{Gluing on simplices}

In this section we describe an inductive method for constructing a
sequence of truncated semi-simplicial schemes converging to \(Z_\bullet\). Here
for any \(i \in \NN\) an \(i\)-\emph{truncated} semi-simplicial object in a
category \(\cC\) is a functor \(\Lambda_{\leq i}^{\textup{op}} \to \cC\), where
\(\Lambda_{\leq i}^{\textup{op}} \) is the full subcategory of
\(\Lambda^{\textup{op}}\) generated by the objects \([j]\) with \(j \leq i\).
Given  an \(i-1\)-truncated semi-simplicial object  \(X_\bullet\) of \(\cC\),
let
\[[i]^2_{<} := \{j, k \in [i] \, | \, j <k\}\] and define two morphisms
\[\delta_{+}, \delta_{-}: X_{i-1}^{[i ]} \to X_{i-2}^{[i]^2_{<}}\] by
\(\delta_{+}(x_0, \dots, x_i) = (d^{i-1}_j(x_k) \, | \, j<k) \) and
\(\delta_{-}(x_0, \dots, x_i) = (d^{i-1}_{k-1}(x_j) \, | \, j<k) \). Assuming \(\cC\) has finite limits we may form the equalizer
\begin{eqnarray}
  \label{eq:coskel-as-equaliz}
  \begin{tikzcd}
    E(X_{\bullet}) := \Eq(\delta_{+}, \delta_{-}) \arrow[r] & X_{i-1}^{[i ]} \arrow[r, shift left, "\delta_{+}"] \arrow[r, shift right, "\delta_{-}"'] & X_{i-2}^{[i]^2_{<}}
  \end{tikzcd}
\end{eqnarray}
one can check that this construction is \emph{functorial} in \(X_\bullet\): indeed if \(Y_\bullet\) is another \(i-1\)-truncated semi-simplicial object then given a morphism \(X_\bullet \to Y_\bullet\) we can form a commutative diagram
\begin{equation}
  \label{eq:func-equaliz}
  \begin{tikzcd}
    E(X_{\bullet}) := \Eq(\delta_{+}, \delta_{-}) \arrow[r] \arrow[d, dashed] & X_{i-1}^{[i ]} \arrow[d] \arrow[r, shift left, "\delta_{+}"] \arrow[r, shift right, "\delta_{-}"'] & X_{i-2}^{[i]^2_{<}} \arrow[d]\\
    E(Y_{\bullet}) := \Eq(\delta_{+}, \delta_{-}) \arrow[r] & Y_{i-1}^{[i ]} \arrow[r, shift left, "\delta_{+}"] \arrow[r, shift right, "\delta_{-}"'] & Y_{i-2}^{[i]^2_{<}}
  \end{tikzcd}
\end{equation}
and obtain a unique morphism on the dashed arrow by functoriality of equalizers. Finally, let \(I\) denote the category \(0 \to 1\) (thought of as the ``unit interval''). An object of \(\cC^I\) is a morphism \(f: X \to Y\) in \(\cC\) and there are 2 functors \(s: \cC^I \to \cC\) defined by \(s(f) = X, t(f) = Y \) (source and target).

\begin{lemma}[{compare with \cite[Vbis, Prop. 5.1.3]{MR0354653},
        \cite[\href{https://stacks.math.columbia.edu/tag/0AMA}{Tag
            0AMA}]{stacks-project}}]
  \label{lem:ssimp-cats-iter-2fib}
  Let \(\cC\) be a category containing finite limits. The functor
  \[ \Phi_i:   \cC^{\Lambda_{\leq i}^{\textup{op}}} \to \cC^{\Lambda_{\leq
      i-1}^{\textup{op}}} \times_{\cC} \cC^I , \] where the right hand side is the 2-fiber product with respect
  to the functors \(E: \cC^{\Lambda_{\leq i-1}^{\textup{op}}} \to \cC \) and
  \(t: \cC^I \to \cC\) that sends an  \(i\)-truncated semi-simplicial object
  \(X_\bullet\) to the pair \((\mathrm{sk}_{i-1}X_\bullet, X_i \to
  E(\mathrm{sk}_{i-1}X))\),  is an equivalence of categories.
\end{lemma}

\begin{proof}
  We first check that \(\Phi_i\) is fully faithful. For faithfulness, note that
  for any 2 \(i\)-truncated semi-simplicial objects \(X_\bullet, Y_\bullet\)
  there is an \emph{injection}
  \begin{equation}
    \label{eq:ssimp-homs-vs-product}
    \Hom_{\cC^{\Lambda_{\leq i}^{\textup{op}}}}(X_\bullet, Y_\bullet)  \inj
    \prod_{j=0}^i \Hom_{\cC}(X_j, Y_j)
  \end{equation}
  since a morphism \(  \alpha: X_\bullet
  \to Y_\bullet\) is equivalent to a sequence of morphisms \(\alpha_i: X_i \to
  Y_i\) commuting with differentials. By the definition of the 2-fiber
  product, the morphism \(\Phi_i(\alpha):\Phi_i(X_\bullet) \to \Phi_i(Y_\bullet) \) induced by \( \alpha \) consists of the morphism
  \(\mathrm{sk}_{i-1}\alpha: \mathrm{sk}_{i-1}X_\bullet \to
  \mathrm{sk}_{i-1}Y_\bullet\), and the commutative diagram
  \[
    \begin{tikzcd}
      X_i \arrow[d, "\alpha_i"] \arrow[r] & E(\mathrm{sk}_{i-1}X) \arrow[d, "E(\alpha)"] \\
      Y_i \arrow[r] & E(\mathrm{sk}_{i-1}Y)
    \end{tikzcd}
  \] This shows that \eqref{eq:ssimp-homs-vs-product} \emph{factors} as
  \begin{equation}
    \Hom_{\cC^{\Lambda_{\leq i}^{\textup{op}}}}(X_\bullet, Y_\bullet)  \xrightarrow{\Phi_i} \Hom_{\cC^{\Lambda_{\leq
      i-1}^{\textup{op}}} \times_{\cC} \cC^I}\big(\Phi_i(X_\bullet),  \Phi_i(Y_\bullet)\big)\to
    \prod_{j=0}^i \Hom_{\cC}(X_j, Y_j)
  \end{equation} hence the first map is injective, or in other words \( \Phi_i \) is faithful. On the other hand given an arbitrary morphism \( \Phi_i(X_\bullet) \to \Phi_i(X_\bullet) \) consisting of a map \(\beta: \mathrm{sk}_{i-1}X_\bullet \to \mathrm{sk}_{i-1}Y_\bullet\), a map \( \gamma: X_i \to Y_i \) and a commutative diagram
  \begin{equation}
    \label{eq:2cat-comm}
    \begin{tikzcd}
      X_i \arrow[d, "\gamma"] \arrow[r] & E(\mathrm{sk}_{i-1}X) \arrow[d, "E(\beta)"] \\
      Y_i \arrow[r] & E(\mathrm{sk}_{i-1}Y)
    \end{tikzcd}
  \end{equation}
  we may verify commutativity of
  \[
    \begin{tikzcd}
      X_i  \arrow[dr, phantom, "(1)"] \arrow[rr, bend left, "d_k^i"]\arrow[r] \arrow[d, "\gamma"] & E(\mathrm{sk}_{i-1}X) \arrow[dr, phantom, "(2)"] \arrow[r, "\mathrm{pr}_k"] \arrow[d, "E(\beta)"] &X_{i-1} \arrow[d, "\beta_{i-1}"]\\
      Y_i  \arrow[rr, bend right, "d_k^i"']\arrow[r] & E(\mathrm{sk}_{i-1}Y)\arrow[r, "\mathrm{pr}_k"] & Y_{i-1}
    \end{tikzcd}
  \] as follows: commutativity of (1) is exactly \eqref{eq:2cat-comm}, and commutativity of (2) can be deduced from that of the left square of \eqref{eq:func-equaliz}. Hence \(  \beta \) and \( \gamma \) define a map \(X_\bullet \to Y_\bullet \) and so \(\Phi_i\) is full.

  Next we show  \(\Phi_i\) is essentially surjective --- this argument is
  inspired by  and closely follows the proof of
  \cite[\href{https://stacks.math.columbia.edu/tag/0186}{Tag
      0186}]{stacks-project}. For this we consider an object of the 2-fiber product
  \( \cC^{\Lambda_{\leq i-1}^{\textup{op}}} \times_{\cC} \cC^I  \) consisting of
  an \(i-1 \)-truncated semi-simplicial object \(X_\bullet\), and object \(Y\)
  and a morphism \(  f: Y \to E(X_{\bullet}) \), and we must prove that there exists
  an \(i\)-truncated semi-simplicial object \(Z_\bullet \) and an isomorphism
  \(\Phi_i(Z_\bullet) \simeq (X_\bullet, f)\). We first let \(  Z_j = X_j \) for
  \( j < i \) and let \(Z(\varphi) = X(\varphi)\) for any \(\varphi: [j'] \to
  [j]\) with \( j' <j < i\). Then we set \(Z_i = Y\), and we must define
  morphisms \(Z(\varphi): Z_i = Y \to X_j = Z_j\) for increasing maps \([j] \to
  [i]\) which are functorial in \(\varphi\), in the sense that for any
  increasing \(\psi: [j'] \to [j]\) the diagram
  \begin{equation}
    \label{eq:func-gluing-on-cells}
    \begin{tikzcd}
      Y \arrow[rr, bend right, "Z(\varphi \circ \psi)"'] \arrow[r, "Z(\varphi)"] & X_j \arrow[r, "X(\psi)"] & X_{j'}
    \end{tikzcd}
  \end{equation}
  commutes (note that the data of \(X(\psi)\) is already included in \(X_\bullet\)).
  We may assume \(j < i\) (otherwise \(\varphi = \mathrm{id}\) and we must set
  \(Z(\varphi) = \mathrm{id}\)), and so \(\varphi\) must factor as
  \[ [j] \xrightarrow{\psi} [i-1] \xrightarrow{\delta^i_k} [i]\] for some \(k\)
  and some \(\psi\). We define \(Z(\varphi) \) to be the composition
  \[ Y \xrightarrow{f} E(X_{\bullet})  \to X_{i-1}^{[i ]} \xrightarrow{\mathrm{pr}_k}
    X_{i-1} \xrightarrow{X(\psi)} X_{j} \] (so in particular we define
  \(Z(\delta^i_k) = \mathrm{pr}_k \circ f =: f_k\)). To verify this
  definition is independent of \(\psi\), suppose that there is another
  factorization
  \[ [j] \xrightarrow{\psi'} [i-1] \xrightarrow{\delta^i_l} [i]\] Note that if
  \(j = i-1\) then \(\psi = \psi' = \mathrm{id}\) and \(k=l\) for trivial
  reasons, so we may assume \(j < i-1\) and in that case \(\varphi\) misses
  \emph{both} \(k\) and \(l\), so we may factor through \([i-2]\) as follows:
  \begin{equation}
    \begin{tikzcd}
      &&  \lbrack i-1\rbrack \arrow[dr, "\delta_k^i"] & \\
      \lbrack j\rbrack \arrow[urr, bend left, "\psi"] \arrow[r, "\rho"] \arrow[drr, bend right, "\psi'"'] & \lbrack i-2\rbrack \arrow[ur, "\delta_{l-1}^{i-1}"] \arrow[dr, "\delta_k^{i-1}"'] && \lbrack i\rbrack \\
      && \lbrack i-1\rbrack \arrow[ur, "\delta_l^i"']
    \end{tikzcd}
  \end{equation}
  By the defining property of the equalizer \(E(X_{\bullet})\), we know \(X(\delta_{j-1}^{i-1})
  \circ f_k = X(\delta_{k}^{i-1}) \circ f_l\), and \[ X(\rho) \circ
    X(\delta_{j-1}^{i-1}) = X(\psi) \text{ and } X(\rho) \circ X(\delta_{k}^{i-1})
    = X(\psi')\] because \(X_\bullet \) is an \(i-1\)-truncated semi-simplicial
  object. It follows that \(X(\psi) \circ f_k = X(\psi') \circ f_l \) as
  desired.

  We now prove to prove the commutativity statement in
  \eqref{eq:func-gluing-on-cells}. Again we may assume \(j < i\), since
  otherwise \(\varphi = \mathrm{id}\) and \(\psi = \varphi \circ \psi\) so
  commutativity is implied by the above proof that the \(Z(\varphi )\)
  are well defined. When \(j < k\) the map \(\varphi\), and hence also \(\varphi
  \circ \psi \) must factor through some \( \delta_k^i: [i-1] \to [i] \) and we
  obtain the following situation:
  \[
    \begin{tikzcd}
      \lbrack j\rbrack \arrow[r, "\psi"] \arrow[rrr, bend right, "\varphi \circ \psi"] & \lbrack j\rbrack \arrow[r, "\rho"] \arrow[rr, bend left, "\varphi" ] & \lbrack i-1\rbrack \arrow[r, "\delta^i_k"] & \lbrack i\rbrack
    \end{tikzcd}
  \] Now \emph{by definition} \(Z(\varphi) = X(\rho) \circ f_k\) and \(Z(\varphi \circ \psi) = X(\rho \circ \psi) \circ f_k\), and since \(X_\bullet\) is an \(i-1 \)-truncated semi-simplicial object \(X(\rho \circ \psi) =X(\psi) \circ X(\rho) \), so that
  \[ X(\psi) \circ Z(\varphi) = X(\psi) \circ X(\rho) \circ f_k = X(\rho
    \circ \psi) \circ f_k = Z(\varphi \circ \psi)\] as claimed.

\end{proof}

\subsection{Common admissible blowups}

Using \cref{lem:ssimp-cats-iter-2fib} to build the semi-simplicial scheme
\(Z_\bullet\) inductively, at each step we encounter the situation of the lemma below.

\begin{lemma}
  \label{lem:u-adm-resolving-indet}
  Suppose
  \[
    \begin{tikzcd}
      X \arrow[d, "\varphi"] & U \arrow[l, "\imath"] \arrow[r, "\jmath"] \arrow[d, "\rho^0"] & Y \arrow[d, "\psi"] \\
      F & E \arrow[l, "f"] \arrow[r, "g"] & G
    \end{tikzcd}
  \]is a commutative diagram of schemes of finite type over a quasi-compact quasi-separated base scheme \(S\), and assume that \(f, g, \varphi\) and \(\psi\) are proper and \(\imath \) and \( \jmath \) are dense open immersions. Then, there is a commutative diagram
  \[
    \begin{tikzcd}
      X \arrow[d, "\varphi"] & Z \arrow[l, "r"] \arrow[r, "s"] \arrow[d, "\rho"] & Y \arrow[d, "\psi"] \\
      F & E \arrow[l, "f"] \arrow[r, "g"] & G
    \end{tikzcd}
  \] where \( r\) and \(s\) are \(U\)-admissible blowups (hence in particular \emph{projective}).

  If in addition \(S\) is a CM-quasi-excellent noetherian scheme and \(U\) is
  regular, then assuming \cref{conj:macify} we may ensure that \(Z\) is Cohen-Macaulay and normal.
\end{lemma}

\begin{proof}
  First, \(X\) and \(E\) are proper over the scheme \(F\), which is
  quasi-compact and quasi-separated since it is of finite type over \(S\). By
  the first part of \cref{lem:common-blp} applied to the map of \(F\)-schemes
  \(\rho^0: U \to E \) defined on the dense open \(U \subseteq  X\), there
  is a \(U\)-admissible blowup \(V_X \to X\) and an \(F\)-morphism \(V_X \to
  E\) extending \(\rho^0\). A similar argument produces a \(U\)-admissible
  blowup \(V_Y \to Y\) and a \(G\)-morphism \(V_Y \to E\) extending \(\rho^0\).
  The current situation is summarized below:
  \[
    \begin{tikzcd}[column sep=large]
      U \arrow[ddr, "\imath"]  \arrow[r] \arrow[drr] & V_Y  \arrow[dd] \arrow[rr] && E \arrow[dr, equals] \arrow[dd, "g", near end] & \\
      &  & V_X   \arrow[rr, crossing over] && E \arrow[dd, "f"] \\
      & Y \arrow[rr, near end, "\psi"] && G & \\
      && X \arrow[from=uuull, crossing over, bend right, "\jmath"'] \arrow[from=uu, crossing over, near end] \arrow[rr, "\varphi"] && F
    \end{tikzcd}
  \]
  Since \(V_X, V_Y\) are \(U\)-admissible blowups of \(X, Y \) respectively,
  they still contain \(U\) as a \emph{dense} open (\cite[comments before Lem. 1.1]{MR2356346}). Note that since \(V_X \to
  X\) is a blowup, \(\varphi\) is proper and \(f\) is proper the morphism \(V_X
  \to E\) is also proper; similarly \(V_Y\) is proper over \(E\). Now applying
  the second part of \cref{lem:common-blp} to \(V_X\) and \(V_Y\) \emph{over}
  \(E\) we obtain a separated finite type morphism \(\rho: Z \to E\), \(U\)
  admissible blowups \(\tilde{V}_X \to V_X\)  and \(\tilde{V}_Y \to V_Y\) and
  open immersions \(\tilde{V}_X \inj Z \hookleftarrow \tilde{V}_Y\) over \(E\)
  such that the diagram
  \[
    \begin{tikzcd}
      U \arrow[r] \arrow[d] & \tilde{V}_Y \arrow[d] \\
      \tilde{V}_X \arrow[r] & Z
    \end{tikzcd}
  \] commutes and \(E = \tilde{V}_X \cup \tilde{V}_Y\). Since \(U\) is dense in both \(\tilde{V}_X \) and \(\tilde{V}_Y\), we see that \(\tilde{V}_X \) and \(\tilde{V}_Y\) are both dense in \(Z\). Then as \(\tilde{V}_X \to Z\) is a dense open immersion of separated finite type \(E\)-schemes where \(\tilde{V}_X\) is \emph{proper} over \(E\), it must be that  \(\tilde{V}_X = Z\); similarly, \(\tilde{V}_Y = Z\) (see also the comments following \cite[Cor. 2.10]{MR2356346}). Finally, we define \(r\) and \(s\) to be the compositions
  \[
    \begin{tikzcd}
      Z \arrow[r, equals] \arrow[rrr, bend left, "r"]&\tilde{V}_X  \arrow[r] &V_X \arrow[r] &X
    \end{tikzcd}  \text{  and  }
    \begin{tikzcd}
      Z \arrow[r, equals] \arrow[rrr, bend left, "s"]&\tilde{V}_Y  \arrow[r] &V_Y \arrow[r] &Y
    \end{tikzcd}
  \]
  Finally if \(S\) is CM-quasi-excellent, then since \(Z\) is of finite type
  over \(S\) it is also CM-quasi-excellent by
  \cite[Rmk.1.5]{CesMac}. By
  hypothesis \(U \subseteq \CM(Z)\), and by \cref{conj:macify} there is a
  \(\mathrm{Reg}(Z)\)-admissible (hence also \(U\)-admissible) blowup \(\tilde{Z} \to Z\)
  such that \(\tilde{Z}\) is Cohen-Macaulay and normal. In this case we replace \(Z\) with \(\tilde{Z}\).
\end{proof}

\begin{lemma}
  \label{lem:ssimp-common-blp}
  Let \(S\) be a quasi-compact quasi-separated base scheme and let
  \begin{equation}
    \label{eq:ssimp-common-blp1}
    \begin{tikzcd}
      X_\bullet \arrow[d] & U_\bullet \arrow[l, "\imath_\bullet"] \arrow[r, "\jmath_\bullet"] \arrow[d] & Y_\bullet \arrow[d] \\
      X_{-1} & U_{-1} \arrow[l, "\imath_{-1}"] \arrow[r, "\jmath_{-1}"] & Y_{-1}
    \end{tikzcd}
  \end{equation}
  be  morphisms of augmented semi-simplicial schemes of finite type over \(S\).
  Assume that all differentials and augmentations of \(X_\bullet\) and
  \(Y_\bullet\) are proper,\footnote{This is equivalent to requiring that
    \(X_\bullet\) is a semi-semi-simplicial object in the category of proper
    \(X_{-1}\)-schemes (and similarly for \(Y_\bullet\)).} and that the morphisms \(X_i \xleftarrow[]{\imath_i}
  U_i \xrightarrow{\jmath_i} Y_i\) are dense open immersions for all \(i\)
  (including \(i = -1\)). If there exists a finite-type \(S\)-scheme \(Z_{-1}\)
  and \(U_{-1}\)-admissible blowups \(X_{-1} \xleftarrow[]{r_{-1}}
  Z_{-1} \xrightarrow{s_{-1}} Y_{-1}\), then there exists an augmented
  semi-simplicial \(S\)-scheme \(Z_\bullet \to Z_{-1}\) together with morphisms
  \begin{equation}
    \label{eq:ssimp-common-blp2}
    \begin{tikzcd}
      X_\bullet \arrow[d] & Z_\bullet \arrow[l, "r_\bullet"] \arrow[r, "s_\bullet"] \arrow[d] & Y_\bullet \arrow[d] \\
      X_{-1} & Z_{-1} \arrow[l, "r_{-1}"] \arrow[r, "s_{-1}"] & Y_{-1}
    \end{tikzcd}
  \end{equation}
  such that for all \(i\) the morphisms \(X_i \xleftarrow[]{r_i}
  Z_i \xrightarrow{s_i} Y_i\) are \(U_{i}\)-admissible blowups (hence in
  particular projective and birational).

  Moreover if \(S\) is a CM-quasi-excellent noetherian scheme, each \(U_i\)
  is regular and \cref{conj:macify} holds, then we may ensure that the \(Z_i\)
  are also Cohen-Macaulay and normal.
\end{lemma}

\begin{proof}
  We construct a sequence of \(i\)-truncated semi-simplicial \(S\)-schemes
  \(\tilde{Z}_{i \bullet}\) converging to \(Z_\bullet\), with the additional
  requirement that the morphisms \(\skel_{i-1} (U_\bullet) \to \skel_{i-1}
  (X_\bullet)\) and \(\skel_{i-1} (U_\bullet) \to \skel_{i-1} (Y_\bullet)\)
  factor through \(\tilde{Z}_{i \bullet}\).\footnote{I \emph{think} that this
    isn't actually an additional restriction, but including it makes the inductive
    step easier.} The \(i=-1\) case is
  included in the hypotheses. At the inductive step we may assume that there is
  an \(i-1\)-truncated semi-simplicial \(S\)-scheme \(\tilde{Z}_{i-1 \bullet}\)
  together with a commutative diagram
  \begin{equation}
    \begin{tikzcd}
      & \skel_{i-1} (U_\bullet) \arrow[dl, "\skel_{i-1}(\imath_\bullet)"'], \arrow[dr, "\skel_{i-1}(\jmath_\bullet)"] \arrow[d, "k_\bullet"] & \\
      \skel_{i-1} (X_\bullet) \arrow[d] & \tilde{Z}_{i-1 \bullet} \arrow[l, "\tilde{r}_{i-1 \bullet}"] \arrow[r, "\tilde{s}_{i-1 \bullet}"] \arrow[d] &\skel_{i-1} (Y_\bullet) \arrow[d] \\
      X_{-1} & Z_{-1} \arrow[l, "r_{-1}"] \arrow[r, "s_{-1}"] & Y_{-1}
    \end{tikzcd}
  \end{equation}
  such that for all \(j < i\) the morphisms \(X_j \xleftarrow[]{\tilde{r}_{i-1,
      j}} \tilde{Z}_{i-1, j} \xrightarrow{\tilde{s}_{i-1, j}} Y_j\) are
  \(U_{j}\)-admissible blowups.
  Letting \(E\) denote the equalizer functor of \cref{lem:ssimp-cats-iter-2fib}, we obtain a commutative diagram of the form
  \begin{equation}
    \label{eq:mps-to-equalizers}
    \begin{tikzcd}
      X_i \arrow[dd, "(X(\delta_k^i))"] & U_i \arrow[l, "\imath_i"] \arrow[r, "\jmath_i"] \arrow[d, "(U(\delta_k^i))"] & Y_i \arrow[dd, "(Y(\delta_k^i))"] \\
      & E(\skel_{i-1} (U_\bullet)) \arrow[d, "E(k_\bullet)"] & \\
      E(\skel_{i-1} (X_\bullet)) & E(\tilde{Z}_{i-1 \bullet}) \arrow[l, "E(\tilde{r}_{i-1 \bullet})"] \arrow[r, "E(\tilde{s}_{i-1 \bullet})"]  &E(\skel_{i-1} (Y_\bullet) )
    \end{tikzcd}
  \end{equation}
  Next, we verify that \eqref{eq:mps-to-equalizers} satisfies the hypotheses of
  \cref{lem:u-adm-resolving-indet}, making repeated reference to the
  constructions in \eqref{eq:coskel-as-equaliz} and \eqref{eq:func-equaliz}.
  Note that the bottom horizontal arrows are proper, since they are obtained as
  limits of the blowup maps \(\tilde{r}_{i-1, j}: \tilde{Z}_{i-1, j} \to X_j\)
  and \(\tilde{s}_{i-1, j}: \tilde{Z}_{i-1, j} \to Y_j\) for \(j=i-1, i-2\). The
  vertical maps on the outside edges are proper since the differentials
  \(X(\delta_k^i): X_i \to X_{i-1}\) and \(Y(\delta_k^i): Y_i \to Y_{i-1}\)
  are proper by hypothesis. Hence applying \cref{lem:u-adm-resolving-indet} we
  obtain a commutative diagram
  \begin{equation}
    \begin{tikzcd}
      U_i \arrow[d, "(U(\delta_k^i))"] \arrow[r, "\imath_i"] \arrow[rr, bend left] \arrow[rrr, bend left, "\jmath_i"] & X_i \arrow[d, "(X(\delta_k^i))"] & Z_i \arrow[l, "\tilde{r}_{i-1, i}"] \arrow[r, "\tilde{s}_{i-1 , i}"] \arrow[d, "\rho"] & Y_i \arrow[d, "(Y(\delta_k^i))"] \\
      E(\skel_{i-1} (U_\bullet)) \arrow[r] \arrow[rr, bend right, "E(k_\bullet)"] \arrow[rrr, bend right] & E(\skel_{i-1} (X_\bullet)) & E(\tilde{Z}_{i-1 \bullet}) \arrow[l, "E(\tilde{r}_{i-1 \bullet})"] \arrow[r, "E(\tilde{s}_{i-1 \bullet})"]  &E(\skel_{i-1} (Y_\bullet) )
    \end{tikzcd}
  \end{equation}
  in which the maps  \(\tilde{r}_{i-1, i}:Z_i \to X_i\) and
  \(\tilde{s}_{i-1, i}:Z_i \to Y_i\) are \(U_i\)-admissible
  blowups. In the case where \(S\) is CM-quasi-excellent, \(U_i \) is regular
  and \cref{conj:macify} holds, we apply
  \cref{lem:u-adm-resolving-indet} to ensure that \(Z_i \) is
  Cohen-Macaulay.

  Now \cref{lem:ssimp-cats-iter-2fib} implies that there is an \(i\)-truncated
  semi-simplicial \(S\)-scheme \(\tilde{Z}_{i\bullet}\) such that
  \(\skel_{i-1}(\tilde{Z}_{i\bullet}) =  \tilde{Z}_{i-1 \bullet}\) and
  \(\tilde{Z}_{i, i} = Z_i\), together with a commutative diagram
  \begin{equation}
    \begin{tikzcd}
      & \skel_{i} (U_\bullet) \arrow[dl, "\skel_{i}(\imath_\bullet)"'], \arrow[dr, "\skel_{i}(\jmath_\bullet)"] \arrow[d, "k_\bullet"] & \\
      \skel_{i} (X_\bullet) \arrow[d] & \tilde{Z}_{i \bullet} \arrow[l, "\tilde{r}_{i \bullet}"] \arrow[r, "\tilde{s}_{i \bullet}"] \arrow[d] &\skel_{i} (Y_\bullet) \arrow[d] \\
      X_{-1} & Z_{-1} \arrow[l, "r_{-1}"] \arrow[r, "s_{-1}"] & Y_{-1}
    \end{tikzcd}
  \end{equation}
  such that for all \(j \leq i\) the morphisms \(X_j
  \xleftarrow[]{\tilde{r}_{i-1, j}} \tilde{Z}_{i-1, j}
  \xrightarrow{\tilde{s}_{i-1, j}} Y_j\) are \(U_{j}\)-admissible blowups.
\end{proof}

\begin{corollary}
  \label{cor:meta-cor}
  With the hypotheses of \cref{lem:ssimp-common-blp}, if in addition the base
  scheme \(S\) is CM-quasi-excellent and noetherian, and all the
  schemes \(X_i\) and \(Y_i\) have pseudorational singularities, all the schemes
  \(U_i \) are regular and \cref{conj:macify} holds, then:
  \begin{enumerate}[(i)]
    \item \label{item:meta-cor1} There is a complex \(\sK\) in the derived
          category \( D_c^b(Z_{-1}) \) together with
          quasi-isomorphisms
          \begin{equation}
            \label{eq:meta-cor1}
            \cone(\sO_{X_{-1}} \to R \epsilon^X_* \sO_{X_\bullet})[-1] \simeq  Rr_{-1*} \sK \text{  and  } Rs_{-1*} \sK \simeq \cone(\sO_{Y_{-1}} \to R \epsilon^Y_* \sO_{Y_\bullet})[-1].
          \end{equation}
    \item \label{item:meta-cor2} There is a complex \(\mathscr{L}\) in the
          derived category \( D_c^b(Z_{-1}) \) together with
          quasi-isomorphisms
          \begin{equation}
            \begin{split}
              &R\sHom (\cone(\sO_{X_{-1}} \to R \epsilon^X_* \sO_{X_\bullet})[-1], \omega_{X_{-1}}^\bullet) \simeq  Rr_{-1*} \sL \text{  and  }\\
              &Rs_{-1*} \mathscr{L} \simeq R\sHom (\cone(\sO_{Y_{-1}} \to R \epsilon^Y_* \sO_{Y_\bullet})[-1], \omega_{Y_{-1}}^\bullet).
            \end{split}
          \end{equation}
  \end{enumerate}
  In the special case where \(X_\bullet\) and \(Y_\bullet\) are the
  semi-simplicial schemes associated to simple normal crossing pairs \(\lsp{X}\)
  and \(\lsp{Y}\) respectively (so that \(X_{-1} = X\) and \(Y_{-1}=Y\)), then there are quasi-isomorphisms
  \begin{equation}
    \label{eq:meta-cor-special}
    \begin{split}
      &\lstrshf{X} \simeq  Rr_{-1*} \sK, \, \, Rs_{-1*} \sK \simeq \lstrshf{Y}, \\
      &\omega_{X}(\Delta_X) \simeq  Rr_{-1*} \sL \text{  and  }
      Rs_{-1*} \mathscr{L} \simeq \omega_{Y}(\Delta_Y)
    \end{split}
  \end{equation}
\end{corollary}

\begin{proof}
  By the ``moreover'' part of \cref{lem:ssimp-common-blp} we may
  ensure \(Z_i\) is Cohen-Macaulay for all \(i\).

  By functoriality, \cref{eq:ssimp-common-blp2} induces commutative diagrams of
  complexes
  \begin{equation}
    \begin{tikzcd}
      \cone(\sO_{X_{-1}} \to R \epsilon^X_* \sO_{X_\bullet})[-1] \arrow[r] \arrow[d, "\alpha^X"]& \strshf{X_{-1}} \arrow[r] \arrow[d, "\beta^X"] &R\epsilon^X_{*} \strshf{X_\bullet} \arrow[r, "+1"] \arrow[d, "\gamma^X"]& \cdots\\
      Rr_{-1*}\cone(\sO_{Z_{-1}} \to R \epsilon^Z_* \sO_{Z_\bullet})[-1] \arrow[r] & Rr_{-1*} \strshf{Z_{-1}} \arrow[r] & Rr_{-1*} R\epsilon^Z_{*} \strshf{Z_\bullet} \arrow[r, "+1"] &\cdots
    \end{tikzcd}
  \end{equation}
  on \(X_{-1} \) and an entirely analogous diagram
  \begin{equation}
    \begin{tikzcd}
      \cone(\sO_{Y_{-1}} \to R \epsilon^Y_* \sO_{Y_\bullet})[-1] \arrow[r] \arrow[d, "\alpha^Y"]& \strshf{Y_{-1}} \arrow[r] \arrow[d, "\beta^Y"] &R\epsilon^Y_{*} \strshf{Y_\bullet} \arrow[r, "+1"] \arrow[d, "\gamma^Y"]& \cdots\\
      Rs_{-1*}\cone(\sO_{Z_{-1}} \to R \epsilon^Z_* \sO_{Z_\bullet})[-1] \arrow[r] & Rs_{-1*} \strshf{Z_{-1}} \arrow[r] & Rs_{-1*} R\epsilon^Z_{*} \strshf{Z_\bullet} \arrow[r, "+1"] &\cdots
    \end{tikzcd}
  \end{equation}
  on \(Y_{-1}\). This motivates the choice:
  \begin{equation}
    \sK := \cone(\sO_{Z_{-1}} \to R \epsilon^Z_* \sO_{Z_\bullet})[-1],
  \end{equation}
  and one can show that \(\sK \) has bounded, coherent cohomology using \cref{cor:bdd-coh}.
  To complete the proof of \cref{item:meta-cor1} it will suffice to show that
  \(\alpha^X\) and \(\alpha^Y\) are quasi-isomorphisms --- we will deal with the
  case of \(\alpha^X\), as that of \(\alpha^Y\) is almost identical.

  As \(X_{-1}\) has rational singularities by hypothesis and by
  \cref{lem:ssimp-common-blp} \(Z_{-1}\) is Cohen-Macaulay and the map \(r_{-1}:
  Z_{-1} \to X_{-1}\) is projective and birational, \cref{thm:sk-ratlsings}
  implies that \(\beta^X\) is a quasi-isomorphism. By commutativity of
  \cref{eq:ssimp-common-blp2}, \(\gamma^X\) can be identified with the morphism
  \begin{equation}
    \label{eq:about-to-use-desc-ss}
    R\epsilon^Y_{*} \strshf{Y_\bullet}  \to R\epsilon^Z_{*} Rs_{\bullet *}  \strshf{Z_\bullet}
  \end{equation}
  and by \cref{lem:desc-ss} the maps on cohomology induced by
  \cref{eq:about-to-use-desc-ss} are the abutment of a map of descent spectral
  sequences; the map of \(E_1\) pages reads
  \begin{equation}
    \label{eq:map-of-some-E1s}
    E_1^{ij}(X_\bullet) = R^j \epsilon^X_{i*} \strshf{X_i} \to R^j \epsilon^X_{i*} Rr_{i*} \strshf{Z_i} = E_1^{ij}(Z_\bullet).
  \end{equation}
  By \cref{thm:sk-ratlsings} again, for each \(i\) the natural map
  \(\strshf{X_i} \to Rr_{i*} \strshf{Z_i}\) is a quasi-isomorphism and from this
  it follows that \cref{eq:map-of-some-E1s} is an isomorphism of \(E_1\) pages;
  this implies \cref{eq:about-to-use-desc-ss} is a quasi-isomorphism. At this
  point we have shown \(\beta^X\) and \(\gamma^X\) are quasi-isomorphisms, and
  by the 5-lemma we conclude that so is \(\alpha^X\).

  For \cref{item:meta-cor2}, we essentially apply the functor \( R\sHom(-,
  \omega_{X_{-1}}^\bullet) \) (resp. \( R\sHom(-,
  \omega_{X_{-1}}^\bullet) \)) to the first (resp. second) quasi-isomorphism of
  \cref{eq:meta-cor1}, and then use Grothendieck duality (\cref{thm:GD}). In more detail:

  Applying \( R\sHom(-,
  \omega_{X_{-1}}^\bullet) \) to the first quasi-isomorphism of
  \cref{eq:meta-cor1} results in a quasi-isomorphism
  \begin{equation}
    R\sHom(\cone(\sO_{X_{-1}} \to R \epsilon^X_* \sO_{X_\bullet})[-1] ,  \omega_{X_{-1}}^\bullet) \simeq  R\sHom(Rr_{-1*} \sK,  \omega_{X_{-1}}^\bullet)
  \end{equation}
  and by \cref{thm:GD} there is a natural quasi-isomorphism
  \begin{equation}
    R\sHom(Rr_{-1*} \sK,  \omega_{X_{-1}}^\bullet) \simeq R\sHom(\sK,  \omega_{Z_{-1}}^\bullet).
  \end{equation}
  This motivates our choice of \( \sL := R\sHom(\sK,
  \omega_{Z_{-1}}^\bullet)\), and establishes the first quasi-isomorphism of
  \cref{eq:meta-cor1}; the proof of the second is similar. Again one can show that \(\sK \) has bounded, coherent cohomology using \cref{cor:bdd-coh}.

  Finally, when \(X_\bullet\) and \(Y_\bullet\) are the semi-simplicial schemes
  associated to simple normal crossing pairs, we know by
  \cref{lem:cech-cx-reg-seq-Cdiv} that there are quasi-isomorphisms
  \begin{equation}
    \begin{split}
      &\cone(\sO_{X} \to R \epsilon^X_* \sO_{X_\bullet})[-1] \simeq \lstrshf{X} \text{  and  } \\
      &\cone(\sO_{Y} \to R \epsilon^Y_* \sO_{Y_\bullet})[-1] \simeq \lstrshf{Y}.
    \end{split}
  \end{equation}
  Combining these with \cref{item:meta-cor1,item:meta-cor2} gives \cref{eq:meta-cor-special}.
\end{proof}

\subsection{Invariance results for cohomology of snc ideal sheaves}

\begin{lemma}
  \label{lem:constr-ssimp-scheme}
  Let \(S\) be an excellent noetherian scheme and let \(\lsp{X}\) and
  \(\lsp{Y}\) be simple normal crossing pairs separated and of finite type over
  \(S\), and let \(X \xleftarrow{r} Z \xrightarrow{s} Y\) be a thrifty proper
  birational equivalence over \(S\). Let \(X_\bullet\) and \(Y_\bullet\) be the
  semi-simplicial schemes associated to \(\lsp{X}\) and \(\lsp{Y}\)
  respectively. Then, there exist morphisms of augmented semi-simplicial schemes
  over \(S\)
  \begin{equation}
    \begin{tikzcd}
      X_\bullet \arrow[d] & U_\bullet \arrow[l, "\imath_\bullet"] \arrow[r, "\jmath_\bullet"] \arrow[d] & Y_\bullet \arrow[d] \\
      X = X_{-1} & U_{-1} \arrow[l, "\imath_{-1}"] \arrow[r, "\jmath_{-1}"] & Y_{-1} = Y
    \end{tikzcd}
  \end{equation}
  such that
  \begin{enumerate}
    \item the morphisms \(X_i \xleftarrow[]{\imath_i} U_i \xrightarrow{\jmath_i}
          Y_i\) are dense open immersions for all \(i\) and
    \item there exists a \(S\)-scheme \(Z_{-1}\)
          and \(U_{-1}\)-admissible blowups \(X_{-1} \xleftarrow[]{r_{-1}}
          Z_{-1} \xrightarrow{s_{-1}} Y_{-1}\).
  \end{enumerate}
  That is, the hypotheses of \cref{lem:ssimp-common-blp} (including the
  non-conjectural hypotheses of its ``moreover'') are satisfied.
\end{lemma}

\begin{proof}
  Set \(\Delta_Z = r^{-1}_* \Delta_X =  s^{-1}_* \Delta_Y\). By
  \cref{lem:common-open} there is a dense open set \(X \hookleftarrow U_X
  \hookrightarrow Z\) (resp \(Z \hookleftarrow U_Y \hookrightarrow Y\))
  containing all generic points of strata of \(\snc(X, \Delta_X)\) and \(\snc(Z,
  \Delta_Z)\) (resp. \(\snc(Y, \Delta_Y)\) and \(\snc(Z, \Delta_Z)\)). Then \(U :=
  U_X \cap U_Y \) is a dense open containing all generic points of strata of
  \(\snc(X, \Delta_X)\), \(\snc(Y, \Delta_Y)\)  and \(\snc(Z, \Delta_Z)\).  Set
  \(\Delta_U := \Delta_Z|_U\), so that \((U, \Delta_U)\) is simple normal crossing
  pair together with thrifty birational (but not necessarily projective)
  morphisms \(\lsp{X} \xleftarrow{r|_U} \lsp{U} \xrightarrow{s|_U} \lsp{Y}\).
  We now let \(X_\bullet\), \(Y_\bullet\) and \(U_\bullet\) be the augmented
  semi-simplicial schemes associated to \(\lsp{X}, \lsp{Y}\) and \(\lsp{U}\) as
  in the discussion at the beginning of \Cref{sec:simpres-specseq}, and consider
  the resulting morphisms
  \begin{equation}
    \begin{tikzcd}
      X_\bullet \arrow[d] & U_\bullet \arrow[l] \arrow[r] \arrow[d] & Y_\bullet \arrow[d] \\
      X_{i-1} = X & U_{i-1} = U \arrow[l] \arrow[r] & Y_{i-1} = Y
    \end{tikzcd}
  \end{equation}
  Since \(U\) contains the generic points of all strata of \(\snc(Z, \Delta_Z)\), the morphisms \(X_{i} \leftarrow U_i \rightarrow Y_i\) are dense
  open immersions for all \(i\), and the differentials and augmentations of
  \(X_\bullet\) and \(Y_\bullet\) are closed immersions, hence proper. Finally
  applying \cref{lem:common-blp} to the collection of open immersions \(U
  \subseteq X,Z\) \emph{over} \(X\), we obtain \(U\)-admissible blowups
  \(\tilde{X}, \tilde{Y}\) of \(X, Y\) respectively, as well as a separated
  finite type \(X\)-scheme \(W\) with open immersions \(\tilde{X}, \tilde{Z}
  \subseteq W\) covering \(W\). Again properness of \(\tilde{X}, \tilde{Y}\)
  over \(X\) forces \(\tilde{X} = \tilde{Z} = W\), hence replacing \(Z\) with
  \(\tilde{Z}\) we can ensure \(r: Z \to X\) is a \(U\)-admissible blowup.
  Repeating this construction with \(Y, Z\) in place of \(X, Z\), we may ensure
  \(s: Z \to Y \) is also a \(U\)-admissible blowup.
\end{proof}

\begin{theorem}
  \label{cor:complex-on-Z}
  With the same hypotheses as \cref{lem:constr-ssimp-scheme}, there are
  quasi-isomorphisms \[ Rf_* \lstrshf{X} \simeq  Rg_* \lstrshf{Y} \text{  and  }
    Rf_* \omega_{X}(\Delta_X) \simeq Rg_* \omega_Y (\Delta_Y). \]
\end{theorem}

\begin{proof}
  By \cref{lem:constr-ssimp-scheme}, the hypotheses of
  \cref{lem:ssimp-common-blp}, hence also its corollary \cref{cor:meta-cor}, are
  satisfied. In particular, since by hypothesis we begin with simple normal
  crossing pairs \(\lsp{X}\) and \(\lsp{Y} \), the ``special case'' of
  \cref{cor:meta-cor} applies, and thus there exist a Cohen-Macaulay scheme
  \(Z_{-1}\), projective birational morphisms \(X \xleftarrow{r_{-1}} Z_{-1}
  \xrightarrow{s_{-1}} Y\), and complexes \(\sK, \sL\) in the derived category
  \(D_c^b(Z)\) together with quasi-isomorphisms
  \begin{equation}
    \label{eq:have-these-qisos}
    \begin{split}
      &\lstrshf{X} \simeq  Rr_{-1*} \sK, \, \, Rs_{-1*} \sK \simeq \lstrshf{Y}, \\
      &\omega_{X}(\Delta_X) \simeq  Rr_{-1*} \sL \text{  and  }
      Rs_{-1*} \mathscr{L} \simeq \omega_{Y}(\Delta_Y).
    \end{split}
  \end{equation}
  Applying the pushforward \(Rf_*\) (resp. \(Rg_* \)) to the first (resp.
  second) of the above quasi-isomorphisms and using the functorial
  quasi-isomorphism \(Rf_* Rr_{-1*} \sK \simeq Rg_* Rs_{-1*} \sK\) we see that
  \begin{equation}
    Rf_* \lstrshf{X} \simeq  Rf_* Rr_{-1*} \sK \simeq  Rg_* Rs_{-1*} \sK \simeq Rg_* \lstrshf{Y}.
  \end{equation}
  Similarly, applying \(Rf_*\) (resp. \(Rg_* \)) to the third (resp.
  fourth)  quasi-isomorphisms of \cref{eq:have-these-qisos} and using the functorial
  quasi-isomorphism \(Rf_* Rr_{-1*} \sL \simeq Rg_* Rs_{-1*} \sL\) we see that
  \begin{equation}
    Rf_* \omega_{X}(\Delta_X)  \simeq  Rf_* Rr_{-1*} \sL \simeq  Rg_* Rs_{-1*} \sL \simeq Rg_* \omega_{Y}(\Delta_Y).
  \end{equation}
\end{proof}

\section{Applications to rational pairs}
\label{sec:app-ratl-pairs}

\begin{definition}[{compare with \cite[Def. 2.78]{MR3057950}}]
  \label{def:ratl-res}
  Let \(\lsp{S}\) be a pair as in \cref{def:pair} and assume
  \(\Delta_S\) is reduced and effective.  A proper birational morphism \(f: X
  \to S\) is a \emph{rational resolution} if and only if
  \begin{enumerate}
    \item \label{item:ratlsings1} \(X\) is regular and the strict transform
          \(\Delta_X := f^{-1}_* \Delta_S\) has simple normal crossings,
    \item \label{item:ratlsings2} the natural morphism \(\sO_S(-\Delta_S) \to
          Rf_* \sO_X(-\Delta_X)\) is a quasi-isomorphism, and
  \end{enumerate}
  letting \(\omega_X = h^{-\dim X}\omega_X^\bullet \) where we use
  \(\omega_X^\bullet = f^! \omega_S^\bullet\) as a normalized dualizing complex
  on \(X\),
  \begin{enumerate}[resume]
    \item \label{item:ratlsings3} \(R^i f_* \omega_X(\Delta_X) = 0\) for
          \(i>0\).
  \end{enumerate}
\end{definition}

In the situation of \cref{def:ratl-res}, the map \(\sO_S(-\Delta_S)
\to Rf_* \sO_X(-\Delta_X)\) appearing in condition \cref{item:ratlsings2} is
Grothendieck dual to a morphism
\begin{equation}
  \begin{split}
    &Rf_* \omega_X^\bullet(\Delta_X) \xrightarrow[(1)]{=}  Rf_* R\sHom_X(\lstrshf{X}, \omega_X^\bullet)  \\
    & \xrightarrow[(2)]{\simeq}  R\sHom_S (Rf_* \lstrshf{X}, \omega_S^\bullet) \xrightarrow[(3)]{}  R\sHom_S(\lstrshf{S}, \omega_S^\bullet)
  \end{split}
\end{equation}
where the equality (1) comes from the fact that \(\Delta_X\) is a Cartier
divisor (\(\lsp{X}\) is snc by hypothesis), the isomorphism (2) comes from
Grothendieck duality and the map (3) is obtained from the morphism of \cref{item:ratlsings2} by applying
the derived functor \(R\sHom_S(-, \omega_S^\bullet)\). As \(X\) is regular and
the dualizing complex \(\omega_X^\bullet\) is normalized  \(h^i \omega_X^\bullet
= 0\) for \(i \neq - \dim X\); in other words, \(\omega_X^\bullet \simeq
\omega_X[\dim X]\). Twisting this equation with the Cartier divisor \(\Delta_X\)
gives \(\omega_X^\bullet (\Delta_X)\simeq \omega_X(\Delta_X)[\dim X]\). If
\(\lstrshf{S} \to Rf_* \lstrshf{X}\) is a quasi-isomorphism, so is
\[ Rf_*  \omega_X(\Delta_X)[\dim X] \simeq  Rf_* \omega_X^\bullet (\Delta_X) \to
  R\sHom_S(\lstrshf{S}, \omega_S^\bullet)\] and taking cohomology sheaves we see
that \(R^{i + \dim X} f_* \omega_X(\Delta_X) \simeq h^iR\sHom_S(\lstrshf{S},
\omega_S^\bullet)\) for all \(i \).

Thus \emph{given} conditions \cref{item:ratlsings1,item:ratlsings2} of
\cref{def:ratl-res}, condition \cref{item:ratlsings3} is equivalent to
Cohen-Macaulayness of the \emph{sheaf} \(\lstrshf{S}\). We record these
observations as a lemma.

\begin{lemma}[{compare with \cite[Cor. 2.73, Props. 2.82-2.23]{MR3057950},
        \cite[Def. 1.3]{kovacsRationalSingularities2022}}]
  \label{lem:ratl-res-alternate}
  With notation and setup as in \cref{def:ratl-res}, the morphism \(f: X \to
  S\) is a rational resolution if and only if
  \begin{enumerate}
    \item \label{item:ratlsingsredux1} \(X\) is regular and the strict transform
          \(\Delta_X := f^{-1}_* \Delta_S\) has simple normal crossings,
    \item \label{item:ratlsingsredux2} the natural morphism \(\sO_S(-\Delta_S)
          \to Rf_* \sO_X(-\Delta_X)\) is a quasi-isomorphism, and
    \item \label{item:ratlsingsredux3} the sheaf \(\lstrshf{S}\) is Cohen-Macaulay.
  \end{enumerate}
\end{lemma}

As illustrated in the examples of \Cref{sec:ex-non-thrifty}, even simple normal
crossing pairs \(\lsp{S}\) may have non-rational resolutions in the absence of
additional thriftiness restrictions, hence the following definition of rational
singularities for pairs.

\begin{definition}
  \label{def:ratl-pairs}
  Let \(\lsp{S}\) be a pair such that \(\Delta_S\) is a reduced effective Weil
  divisor. Then, \(\lsp{S}\) is \emph{resolution-rational}  if and only if it
  has a thrifty rational resolution.
\end{definition}

\subsection{Rational resolutions of pairs are all-for-one}

In the case where \(S\) is a normal variety over a field of characteristic 0, it
is known that if \(\lsp{S}\) has a thrifty rational resolution then \emph{every}
thrifty resolution is rational \cite[Cor. 2.86]{MR3057950}. The proof of this
fact shows more generally that if \(f: X \to S\) and \(g: Y \to S\) are
thrifty resolutions, then there are isomorphisms \(R^i f_* \lstrshf{X} \simeq
R^i g_* \lstrshf{Y}\) for all \(i\). This remains true in arbitrary
characteristic.

\begin{theorem}[{\cite[Cor. 2.86]{MR3057950} in characteristic 0}]
  \label{thm:all41}
  Let \(\lsp{S}\) be a pair such that \(\Delta_S\) is a reduced effective Weil
  divisor, let \(f: X \to S\) and \(g: Y \to S\) be thrifty resolutions and assume \cref{conj:macify} holds.
  Then there is are quasi-isomorphisms \(Rf_* \lstrshf{X} \simeq Rg_*
  \lstrshf{Y}\) and \( Rf_* \omega_X (\Delta_X) \simeq Rg_* \omega_Y(\Delta_Y)\) . In particular, \(f\) is a rational resolution if and only if
  \(g\) is.
\end{theorem}

Note that this includes \cref{thm:smth-to-smth} as a special case: indeed, if
\(\lsp{S}\) is a simple normal crossing pair then given any thrifty resolution
\(f: X \to S\) we may choose \(g\) to be the identity.

\begin{proof}
  By \cref{lem:common-open}, there are dense open immersions \(S \hookleftarrow
  U_X \hookrightarrow X\) and \(S \hookleftarrow
  U_Y \hookrightarrow Y\) such that \(U_X\) (resp. \(U_Y\)) contains all strata
  of \(\snc \lsp{S}\) and \(\lsp{X}\) (resp. \(\snc \lsp{S}\) and
  \(\lsp{Y}\)). Then \(U:= U_X \cap U_Y\) also contains all strata of \(\snc
  \lsp{S}\) --- moreover since \(f\) and \( g\) are thrifty, the strata of
  \(\lsp{X}\)  and \(\lsp{Y}\) are in one-to-one birational correspondence with
  those of \(\lsp{S}\), so it remains true that \(U\) contains all strata of
  \(\lsp{X}\)  and \(\lsp{Y}\). Replacing \(U\) with \(U \cap \snc \lsp{X}\), we
  may assume \((U, \Delta_U := \Delta_S \cap U)\) is an snc pair. We now have
  morphisms \(\imath : U \inj X\), \(\jmath: U \inj Y\) which are thrifty and birational, but not
  necessarily proper.

  Now let \(X_\bullet \to X_{-1} =: X\), \(Y_\bullet \to Y_{-1} =: Y\) and
  \(U_\bullet \to U_{-1} =: U\) be the augmented semi-simplicial schemes
  associated to these simple normal crossing pairs. The inclusions \(\imath\)
  and \(\jmath\) induce a diagram as in \eqref{eq:ssimp-common-blp1}; we proceed
  to verify that the hypotheses of \cref{lem:ssimp-common-blp} are satisfied.
  All schemes in sight are defined over the noetherian and hence quasi-compact
  quasi-separated \(S\). The differentials and augmentations are all closed
  immersions and hence proper, and thriftiness of  \(\imath\) and \(\jmath\)
  implies that the morphisms \(X_i \xleftarrow[]{\imath_i} U_i
  \xrightarrow{\jmath_i} Y_i\) are dense open immersions for all \(i\). Applying
  \cref{lem:common-blp} to the collection of \(S\)-schemes \(S, X\) and \(Y\)
  with the common dense open \(U\) gives a common \(U\)-admissible blowup \(X
  \leftarrow Z \rightarrow Y\). Finally (for the moreover part of the lemma)
  \(S\) is excellent by hypothesis, the \(U_i\) are regular and we are assuming
  \cref{conj:macify} holds.

  Applying  \cref{lem:ssimp-common-blp} and \cref{cor:meta-cor} we obtain a
  Cohen-Macaulay scheme \(Z_{-1}\), projective birational morphisms \(X
  \xleftarrow{r_{-1}} Z_{-1} \xrightarrow{s_{-1}} Y\), and complexes \(\sK,
  \sL\) in the derived category \(D_c^b(Z)\) together with quasi-isomorphisms
  \begin{equation}
    \label{eq:have-these-qisos}
    \begin{split}
      &\lstrshf{X} \simeq  Rr_{-1*} \sK, \, \, Rs_{-1*} \sK \simeq \lstrshf{Y}, \\
      &\omega_{X}(\Delta_X) \simeq  Rr_{-1*} \sL \text{  and  }
      Rs_{-1*} \mathscr{L} \simeq \omega_{Y}(\Delta_Y).
    \end{split}
  \end{equation}
  As was the case for \cref{cor:complex-on-Z}, pushing these quasi-isomorphisms
  forward with \(Rf_* \) and \(Rg_* \) and using the functorial
  quasi-isomorphisms \(Rf_* Rr_{-1*} \sK \simeq R_g* Rs_{-1*} \sK  \) and \(Rf_*
  Rr_{-1*} \sL \simeq R_g* Rs_{-1*} \sL \) completes the proof.
\end{proof}

\subsection{Semi-simplicial versus thrifty resolutions}

We can also deduce the following lemma. While it's statement is quite verbose,
it may be of interest as it opens the possibility of defining
rational singularities of pairs without thrifty resolutions.
\begin{lemma}
  \label{lem:ssimp-ratl-pairs}
  Let \(\lsp{S}\) be a pair, with \(\Delta_{S}\) reduced and effective, and
  suppose \(\lsp{S}\) admits a thrifty resolution. Let \(\epsilon^S:
  S_\bullet\to S\) be the associated semi-simplicial scheme and suppose there
  exists an augmented semi-simplicial scheme \(\epsilon^Y: Y_\bullet \to Y\)
  with proper differentials and augmentation,
  together with a map \(g_\bullet: Y_\bullet \to S_\bullet\)  such that each
  \(Y_i\) has rational singularities in the sense of \cite[Def. 1.3]{kovacsRationalSingularities2022} and all of the morphisms \(Y_i \to S_i\) are admissible blowups
  for \(U_{Y, i} \subseteq S_i\), where \(U_Y \subseteq S\) is a dense open containing
  the generic points of all strata of \(\snc\lsp{S}\) and \(U_{Y, \bullet}\) is the
  semi-simplicial scheme associated to \((U_Y, U_Y \cap \Delta_S)\). Let
  \(\mathscr{K} = \cone(\strshf{Y} \to R\epsilon^Y_* \strshf{Y_\bullet})[-1]\)
  (this is a complex in the derived category of \(Y \)).
  Then, \(\lsp{S}\) is a rational pair if and only the sheaf \(\lstrshf{S}\)
  is Cohen-Macaulay and the natural map \(\lstrshf{S} \to Rg_*
  \mathcal{K}\) is quasi-isomorphism.
\end{lemma}

\begin{proof}
  Let \(f: X \to S\) be a thrifty resolution. We observe that the first two
  paragraphs of the proof of \cref{thm:all41} remain valid when \(Y_\bullet\) is
  defined as in the lemma (as opposed to being obtained as the associated
  semi-simplicial scheme of an snc pair \(\lsp{Y}\)). Since the \(Y_i\) are
  assumed to have rational singularities in the sense of
  \cite[Def. 1.3]{kovacsRationalSingularities2022} the spectral sequence argument
  appearing in \cref{cor:meta-cor} remains valid and we conclude
  \begin{equation}
    \label{eq:ssimpratlpairs}
    Rf_* \lstrshf{X} \simeq Rg_*  \mathcal{K}.
  \end{equation}
  Now, if \(\lsp{S}\) is a rational pair, by \cref{def:ratl-res} and \cref{lem:ratl-res-alternate} we conclude that
  \(\lstrshf{S} \simeq Rf_*\lstrshf{X}\) and \(\lstrshf{S}\) is Cohen-Macaulay;
  then from  \cref{eq:ssimpratlpairs} we conclude \(\lstrshf{S} \simeq Rg_*
  \mathscr{K}\). On the other hand, if \(\lstrshf{S}\) is Cohen-Macaulay and
  \(\lstrshf{S} \simeq Rg_* \mathscr{K}\) \cref{eq:ssimpratlpairs} shows \(\lstrshf{S} \simeq
  Rf_*\lstrshf{X}\) and \cref{def:ratl-res} and \cref{lem:ratl-res-alternate}
  imply \(f\) is a rational resolution.
\end{proof}

In characteristic 0 where we know resolutions exist, we can construct a
\(Y_\bullet \) as in the statement of the lemma with each \(Y_i\) \emph{smooth},
using the methods of \Cref{sec:constructions}, but substituting Hironaka's
strong resolution of singularities for \cref{conj:macify}. \emph{Unlike} an
augmented semi-simplicial resolution obtained from a thrifty resolution \(f: X
\to S\), the differentials of \(Y_\bullet\) need not be closed embeddings. To
illustrate this distinction in the simplest possible situation, suppose \(S\) is a smooth variety over \(\CC\) and
\(D \subset S\) is a (not necessarily smooth) irreducible divisor. In this case a thrifty resolution is simply an embedded resolution of
\(D\). On the other hand, an
augmented semi-simplicial resolution \(g_\bullet Y_\bullet \to S_\bullet\) can be obtained as
\(\epsilon: \tilde{D} \to S\) where \(\pi: \tilde{D} \to D\) is any resolution of \( D\)
and \(\epsilon\) is the composition \(\tilde{D} \to D \inj S\). Unpacking the
definition of \(\mathscr{K}\) we see that it is the shifted cone of the morphism
\(\strshf{S} \to R\pi_* \strshf{\tilde{D}}\), and hence coincides with
\(\strshf{S}(-D)\) precisely when \(R\pi_* \strshf{\tilde{D}} =\strshf{D}\).
This recovers the fact that when \(S\) is smooth and \(D\) is irreducible, \((S,
D)\) is a rational pair if and only if \(D\) has rational singularities
\cite[Rmk. 2.85]{MR3057950}.

\nocite{EricksonLindsay2014Dior}
\printbibliography

\clearpage

\appendix

\section{(Non-)examples of thrift}
\label{sec:ex-non-thrifty}

In this section we work over a field \(k\). Our first example is not new, and likely served as the original motivation for
considering thrify morphisms.

\begin{example}
  \label{ex:blp-2lines}
  Let \(S = \AA^2_{xy}\) and \(\Delta = V(xy)\). Then \(f:
  X = \Bl_0 S \to S\) is neither thrifty nor rational.
  Indeed, letting \(D_1 = V(x), D_2 = V(y)\) we see that \(\Delta\) is the union
  of the 2 lines \(D_1, D_2\) meeting at the origin. Let \(\tilde{D}_i =
  f^{-1}_* D_i  \) be the strict transforms, \(E = f^{-1}(0)\) the exceptional
  divisor, and \(\tilde{\Delta} = \tilde{D}_1+ \tilde{D}_2\) 
  .
  The map \(f: X \to S\) fails to be thrifty since it is not an isomorphism over
  the stratum \({0} = D_1 \cap D_2\) of \((S, \Delta)\). We will calculate
  cohomology to show \(f\) isn't rational either. 
  
  Since \(S = \AA^2_{xy}  \) is affine, we can identify the sheaves \(R^if_*
  \sO_{X}( -\tilde{\Delta})\) as the sheaves associated to the \(k[x, y]\)-modules
  \(H^i(X, \sO_{X}(-\tilde{\Delta}))\). Observe that \(X\) can be identified with
  the geometric line bundle \(\Spec_{\PP^1} \Sym \sO_{\PP^1}(1)\)
  associated to \(\sO_{\PP^1}(-1)\). Under this identification, the
  projection \(\pi: \Spec_{\PP^1} \Sym \sO_{\PP^1}(1) \to \PP^1\) corresponds to the
  composition \(\Bl_0 S \subseteq \AA^2 \times \PP^1 \to \PP^1\), and the blowup
  map \(f: \Spec_{\PP^1} \Sym \sO_{\PP^1}(1) \to \AA^2\) corresponds to the
  natural
  map 
  \[ \Spec_{\PP^1} \Sym \sO_{\PP^1}(1) \to \Spec_{k} H^0(\PP^1,  \Sym
  \sO_{\PP^1}(-1)) =  \Spec_{k} k[x, y] = \AA^2 \]
  Hence \(\tilde{\Delta} = \pi^* (0 + \infty)\). Now since \(\pi\) is affine its
  Leray spectral sequence degenerates to give 
  \[ \begin{split}
    H^i(X, \sO_{X}(-\tilde{\Delta})) &=  H^i(\PP^1,
  \pi_*\sO_{X}(-\tilde{\Delta}))  \text{ and via projection formula }\\
  \pi_*\sO_{X}(-\tilde{\Delta}) &= \pi_*\sO_{X}(-\pi^* (0 + \infty)) =
  (\pi_*\sO_{X}) (-0 - \infty)
  \end{split}\]
  By the correspondence between affine schemes and sheaves of algebras,
  \[\pi_*\sO_{X} = \Sym \sO_{\PP^1}(1) = \bigoplus_{d\geq 0}\sO_{\PP^1}(d)\]
  Hence \(H^i(X, \sO_{X}(-\tilde{\Delta})) = \oplus_{d\geq 0} H^i(\PP^1,
  \sO_{\PP^1}(d-2))\). In particular, when \(i=1\) and \(d=0\), we see \(H^1(X,
  \sO_{X}(-\tilde{\Delta})) = H^1(\PP^1, \sO_{\PP^1}(-2)) \simeq k \) by
  \cite[Thm. III.5.1]{MR0463157}. 
\end{example}

An elaboration of \cref{ex:blp-2lines} shows in general that if \((S,
\Delta)\) is a simple normal crossing pair and \(Z \subseteq S\) is a stratum,
then \(f: X = \Bl_Z S \to S\) fails to be thrifty. Localizing at the generic
point \(\eta \in Z\) we can reduce to the case where \(Z\) is replaced by a
closed point \(\eta \in S\) and \(\Delta = V(x_1 \cdot x_2 \cdots x_n)\) where
\(x_1,\cdots, x_n \in \mathfrak{m}_\eta\) is a regular system of parameters.
Then the long exact sequence obtained by pushing forward \(\sO_X(-\tilde{\Delta}
- E) \to \sO_X(-\tilde{\Delta}) \to \sO_E(-\tilde{\Delta}|_E)  \) ends in 
\[  R^{n-1}f_* \sO_X(-\tilde{\Delta} - E) \to R^{n-1}f_* \sO_X(-\tilde{\Delta})
\to R^{n-1}f_* \sO_E(-\tilde{\Delta}|_E) \to R^{n}f_* \sO_X(-\tilde{\Delta} - E)
=0\] where the vanishing on the right holds since the maximal fiber dimension of
\(f\) is \(n-1\) \cite[Cor. III.11.2]{MR0463157}. Thus \(R^{n-1}f_*
\sO_X(-\tilde{\Delta}) \to R^{n-1}f_* \sO_E(-\tilde{\Delta}|_E) = H^{n-1}(E,
\sO_E(-\tilde{\Delta}|_E))\) is surjective, and identifying \(E\) with the
projectivized Zariski  tangent  space \(\PP(TS_{\eta})\) with homogeneous
coordinates \(x_1,\dots, x_n\) and \(\tilde{\Delta}|_E\) with \(V(\prod_i x_i)\)
shows \(H^{n-1}(E, \sO_E(-\tilde{\Delta}|_E))\simeq H^{n-1}(\PP^{n-1},
\sO_{\PP^{n-1}}(-n))\simeq k\). For related discussion see \cite[p. 86]{MR3057950}.

The next example answers (in the affirmative!) a question of Erickson
\cite[p.2]{Erickson2014DeformationIO} and Prelli \cite[p.3]{Prelli2017OnRD}
about whether there exists a resolution which is rational but not thrifty. In
fact, we give such an example where the underlying pair \((S, \Delta)\) is rational.

\begin{example}
  \label{ex:not-thrifty-but-rational}
  Let \(S = V(xy-zw) \subseteq \AA^4_{xyzw}\), \(D_0 = V(x, w)\) and
  \(D_{\infty} = V(y, z)\); finally let \(C_\infty = V(w, y)\). We can identify \(S = C(\PP^1 \times \PP^1)\) as the
  affine cone over the Segre embedding \(\PP^1_{s} \times \PP^1_{t} \inj
  \PP^3_{xyzw} \) given by 
  \begin{equation}
    \label{eq:segre-matrices}
    \begin{bmatrix}
      x & w \\
      z & y
    \end{bmatrix} = 
    \begin{bmatrix}
      s_0 \\
      s_1   
    \end{bmatrix} \begin{bmatrix}
      t_0 & t_1
    \end{bmatrix} = 
    \begin{bmatrix}
      s_0 t_0 & s_0 t_1 \\
      s_1 t_0 & s_1 t_1 \end{bmatrix}
    \end{equation}
    Hence \(D_0 = C(\{0\}\times \PP^1) \),
    \(D_\infty = C(\{\infty\}\times \PP^1)\) and \(C_\infty = C(\mathbb{P}^1
    \times \{\infty\})\). 
    
    Let \(\Delta = D_0 +D_\infty + C_\infty\). Note that \(\Delta\) is \emph{not}
    Cartier, as it is not linearly equivalent to any multiple of \(C(\{0\}\times
    \PP^1) +  C(\PP^1 \times \{0\} )\) (here \(\{0\}\times \PP^1 +\PP^1 \times
    \{0\} \) is a hyperplane section of the Segre embedding) --- see e.g.
    \cite[Ex. II.6.3]{MR0463157},\cite[Prop. 3.14]{MR3057950}. Since \(K_S\)
    \emph{is} \(\QQ\)-Cartier, it follows that the pair \((S, \Delta=D_0
    +D_\infty)\) is not \(\QQ\)-Gorenstein --- in particular it isn't dlt, so we
    are not at risk of violating \cite[Thm. 2.87]{MR3057950} which implies that
    a resolution of a dlt pair is thrifty \emph{if and only if} it is rational.
    
    Now let \(f: X= \Bl_{D_0} S \to S\) be the blowup at \(D_0\), let
    \(\tilde{D}_i = f^{-1}_* D_i\) for \(i = 0, \infty\) and \(\tilde{C}_\infty
    = f^{-1}_* C_\infty\), and let \(\tilde{\Delta} = \tilde{D}_0 +
    \tilde{D}_\infty + \tilde{C}_\infty\). The map \(f\) is a small resolution
    of \(S\) (as mentioned in \cite[Ex. 2.7]{MR1658959}). This means we are
    not at risk of violating \cite[Prop. 1.6]{Erickson2014DeformationIO} which
    states that if a log resolution of a pair is rational then it is thrifty. Indeed, the ambient
    blowup is described as 
    \[ \Bl_{D_0}\AA^4 \subseteq \{(x, y, z, w), [u,v] \, | \, (x, w) \propto
    (u,v)\} \subseteq \AA^4 \times \PP^1_{uv}  \]
    so on the \(D(u)\) patch \((x, w) = \lambda (1, v)\) and 
    \[ xy - zw = \lambda y - z \lambda v = \lambda (y -zv) \] Since \(V(\lambda)\)
    is the exceptional divisor we see the strict transform \(X \subseteq
    \Bl_{D_0}\AA^4\) of \(S\) is \(V(y-zv)\) on the \(u=1\) patch --- this is
    smooth as it's a graph. By symmetry in \(x, w\), we conclude \(X\) is smooth.
    
    Even better, this allows us to parametrize \(X \cap D(u)\) with coordinates
    \(z, \lambda, v\):
    \begin{equation}
      \label{eq:segre-patch}
      \begin{split}
        &\AA^3_{z \lambda v} \simeq \Bl_{D_0} S \cap D(u) \subseteq D(u) \simeq
        \AA^5_{xyzwv} \\
        & \text{ sending } (z, \lambda, v) \mapsto (\lambda, zv, z,
        \lambda v, v) = (x,y,z,w,v)
      \end{split}
    \end{equation}
    So in particular the restriction of \(f\) looks like \((z, \lambda, v)
    \mapsto (\lambda, zv, z, \lambda v)\) and we see that the exceptional locus
    is the \(v\)-axis. In this coordinate patch the strict transforms
    \(\tilde{D}_0\) and \(\tilde{D}_\infty\) are \(V(\lambda)\) and \(V(z)\)
    respectively, which \emph{intersect along the \(v\)-axis \(V(\lambda, z)\)}!
    Thus \(\tilde{\Delta}\) has a stratum in \(\Ex(f)\) and \(f\) isn't thrifty.
    We also see that on this patch \(\tilde{C}_\infty = V(v)\). As a
    philosophical aside, the blowup coordinates \([u, v]\) correspond to \([x,
    w] = [s_0 t_0, s_0 t_1] = [t_0, t_1] \) as long as \(s_0\neq 0\), so \(\Ex
    f\) can be viewed as a copy of the \(\PP^1_t\) appearing in \(D_0 = C(\{0\}
    \times \PP^1_t)\) 
    .
    
    
    To show that \(f\) \emph{is} in fact a rational resolution we will use an
    alternative description of \(X\). Starting with the ample invertible sheaf
    \(\sO_{\PP^1 \times \PP^1}(1,1)\) we have natural morphisms of relative
    spectra
    \begin{equation}
      \label{eq:mps-rel-specs}
      \Spec_{\PP_s^1 \times \PP_t^1} \Sym \sO_{\PP^1 \times \PP^1}(1,1) \to
      \Spec_{\PP_t^1} \Sym \mathrm{pr}_{t*} \sO_{\PP^1 \times \PP^1}(1,1) \xrightarrow{f'}
      \Spec_{k} H^0(\Sym \sO_{\PP^1 \times \PP^1}(1,1) )
    \end{equation}
    where \(\pr_t : \PP_s^1
    \times \PP_t^1 \to \PP_t^1\) is the projection. It is well known that the
    scheme on the left can be identified with the blowup \(\Bl_0 S\), and the
    scheme on the right is \(S\). 
    \begin{claim}
      There is an isomorphism of \(S \times \PP^1\)-schemes 
      \[X = \Bl_{D_0}S \simeq \Spec_{\PP_t^1} \Sym \mathrm{pr}_{t*} \sO_{\PP^1
      \times \PP^1}(1,1) \]
    \end{claim}
    This can be proved via the universal property. On the other hand, at
    least when \(k\) is algebraically closed, a quick, dirty and more illuminating
    proof is possible: we \emph{have} a morphism  \( (f', \pi): \Spec_{\PP_t^1}
    \Sym \mathrm{pr}_{t*} \sO_{\PP^1 \times \PP^1}(1,1) \to S \times \PP_t^1 \):
    the first factor is the second map of \eqref{eq:mps-rel-specs}, the second is
    the canonical projection 
    \[\pi : \Spec_{\PP_t^1} \Sym \mathrm{pr}_{t*} \sO_{\PP^1 \times \PP^1}(1,1)
    \to  \PP_t^1\] from the relative \(\Spec\) construction. \(X \subseteq S
    \times \PP_t^1\) by construction, and we can check \(\varphi\) maps the
    \(k\)-points of \(\Spec_{\PP_t^1} \Sym \mathrm{pr}_{t*} \sO_{\PP^1 \times
    \PP^1}(1,1) \) bijectively onto those of \(X\). Indeed, the fiber of
    \(\Spec_{\PP_t^1} \Sym \mathrm{pr}_{t*} \sO_{\PP^1 \times \PP^1}(1,1)\) over
    \(t \in \PP^1_t\) can be described as follows: Note by projection formula
    \begin{equation}
      \begin{split}
        &\mathrm{pr}_{t*} \sO_{\PP^1 \times \PP^1}(1,1) \simeq H^0(\PP^1_s ,
    \sO_{\PP^1_s }(1))\otimes_k \sO_{\PP^1_t }(1),  \\
    \text{  so  } &\Sym
    \mathrm{pr}_{t*} \sO_{\PP^1 \times \PP^1}(1,1) \simeq \Sym H^0(\PP^1_s ,
    \sO_{\PP^1_s }(1))\otimes_k \sO_{\PP^1_t }(1) 
      \end{split}
    \end{equation}
    Explicitly \(\Sym  H^0(\PP^1_s , \sO_{\PP^1_s }(1))\otimes_k \sO_{\PP^1_t }(1)  = \oplus_d k[s_0,  s_1]_d \otimes \sO_{\PP^1_t }(d) = k[s_0,s_1] \times  \Sym \sO_{\PP^1_t }(1)
    \) where \(\times \) denotes the product of graded rings of \cite[Ex. II.5.11]{MR0463157}  and hence for a \(k\)-point \(t\),
    \[ \pi^{-1}(t) \simeq \Spec k[s_0, s_1], \text{ so that } f'|_{\pi^{-1}(t)} :
    \pi^{-1}(t) \to S \] is a map \(\AA^2_{s_0s_1} \to S \subseteq \AA^4_{xyzw}\).
    Writing down the map of algebras corresponding to \(f'\) shows that it is none
    other than the linear transformation of \eqref{eq:segre-matrices}. Finally,
    referencing \eqref{eq:segre-patch} we see that the fibers of \(X \to \PP^1_t\)
    have the same description.\footnote{In slogan form: \(X = \Bl_{D_0}S\) is a
    pencil of 2-planes on \(S\) corresponding to the pencil of rulings \(\PP^1_s
    \times \{t\} \subseteq \PP^1\times \PP^1\).}
    
    Using the claim, we proceed as in \cref{ex:blp-2lines} using degeneration of
    the Leray spectral sequence for the affine map \(\pi: X \to \PP_t^1\) to
    calculate
    \[H^i(X, \sO_X(-\tilde{\Delta})) = H^i(\PP_t^1, \pi_*
    \sO_X(-\tilde{\Delta}))\] On \(\PP_t^1\), noting that \(\tilde{C}_\infty =
    \pi^*(\infty)\), the projection formula gives 
    \begin{equation}
      \label{eq:proj-formula-atiyah-ex}
      \pi_* \sO_X(-\tilde{\Delta}) = (\pi_* \sO_X(-\tilde{D}_0 -
      \tilde{D}_\infty))(-\infty) 
    \end{equation}
    and \(\pi_* \sO_X(-\tilde{D}_0 -
    \tilde{D}_\infty)\subseteq \pi_*
    \sO_X\)  is the sheaf of ideals \((s_0\cdot s_1) \subseteq k[s_0,s_1] \times
    \Sym \sO_{\PP^1_t }(1)\). Letting \((s_0\cdot s_1)_d \subseteq k[s_0,
    s_1]\) denote the \(d\)-th graded part, we see  
    \[\pi_* \sO_X(-\tilde{\Delta}) = \bigoplus_{d\geq 0} (s_0\cdot  s_1)_d
    \otimes_k \sO_{\PP_t^1}(d-1)\]
    where the ``-1'' comes from the twist ``\((-\infty)\)'' in
    \eqref{eq:proj-formula-atiyah-ex}. This yields:
    \begin{equation}
      \begin{split}
        &H^i(\PP_t^1, \pi_* \sO_X(-\tilde{\Delta})) = \bigoplus_{d\geq 0} (s_0\cdot  s_1)_d \otimes_k H^i(\PP_t^1, \sO_{\PP_t^1}(d-1)) \\
        &= \begin{cases}
          \bigoplus_{d\geq 0} (s_0\cdot
          s_1)_d \otimes_k (t_1)_d \subseteq k[s_0, s_1] \times
          k[t_0, t_1] =H^0(S, \sO_S) & \text{ if } i =0\\
          0 & \text{ if } i =1
        \end{cases}
      \end{split}
    \end{equation}
    the key point being that  \(H^1(\PP_t^1, \sO_{\PP_t^1}(d-1)) = 0\) for \(d\geq
    0\). This calculation shows \(f_*\sO_X(-\tilde{\Delta}) = \sO_S(-\Delta) \)
    (this holds for more general reasons, namely \(S\) is normal \cite[Lem. 2.1]{Prelli2017OnRD}) and
    \(R^1f_*\sO_X(-\tilde{\Delta}) = 0\).
    
    Finally, \((S, \Delta)\) \emph{is} a rational pair, as a consequence of the
    theorem below --- this was the main reason for including the additional
    divisor \(C_{\infty}\). If we had left it out, the above calculations would
    still show that \(f: X \to (S, D_0 + D_\infty)\) is a non-thrifty rational
    resolution, however the pair \((S, D_0 + D_\infty)\) isn't rational (also by
    the theorem below).    
    \begin{theorem}[{\cite[Thm. 3.2]{Prelli2017OnRD}}]
      \label{lem:ratl-cone-pairs}
      Let \((Y,B)\) be a pair such that \(Y\) is a normal variety over \(k
      \) and \(B \) is a reduced effective Weil divisor on \(Y\) (for
      example a simple normal crossing pair) and let \(\sL\) be an ample
      invertible sheaf on \(Y\). Let \((CY, CB) \) be the abstract affine
      cone over \((Y,B)\) with respect to \(\sL\): \(CY = \Spec_k H^0(Y,
      \Sym \sL)\) and \(CB\) is the image of \(\Spec_k H^0(B, \Sym
      \sL|_B) \to \Spec_k H^0(Y, \Sym \sL) = CY\) with its reduced
      subscheme structure. Then \((CY, CB) \) is a rational pair if and
      only if \((Y,B)\) is a rational pair and 
      \[H^i(Y, \sL^d(-B)) = 0 \text{ for } i >0, d \geq 0\]
    \end{theorem}

    Applying the theorem to \(Y = \PP^1 \times \PP^1\) with the divisor \(B =
    \{0, \infty\} \times \PP^1 + \PP^1 \times \{\infty\}\) 
     which has associated invertible sheaf \(\sO_Y(B) \simeq
    \sO_{\PP^1 \times \PP^1}(2,1)\), together with the ample invertible sheaf
    \(\sL = \sO_{\PP^1 \times \PP^1}(1,1)\) we calculate (using K\"unneth)
    \begin{equation}
      \label{eq:check-its-ratl}
      H^i(Y, \sL^d(-B)) = H^i(\PP^1 \times \PP^1, \sO_{\PP^1 \times \PP^1}(d-2,d-1)) = \bigoplus_{j+k=i} H^j(\PP^1, \sO_{\PP^1}(d-2)) \otimes_k H^k(\PP^1, \sO_{\PP^1}(d-1))
    \end{equation}
    Noting that \(H^k(\PP^1, \sO_{\PP^1}(d-1)) = 0\) for \(k>0\) and \(d\geq
    0\), we see that \(H^2(Y, \sL^d(-B)) = 0\) for \(d\geq0\), and 
    \[  H^1(Y, \sL^d(-B)) = H^1(\PP^1, \sO_{\PP^1}(d-2)) \otimes_k H^0(\PP^1,
    \sO_{\PP^1}(d-1)) \]
    Now \(H^1(\PP^1, \sO_{\PP^1}(d-2)) =0\) for \(d \neq 0\), but \(H^0(\PP^1,
    \sO_{\PP^1}(-1)) = 0 \), so the tensor product is always 0.
  \end{example}

The last example of this section shows that even when \((S, \Delta)\) is a
simple normal crossing pair and \(f: X \to S\) is a \(U\)-admissible blowup for
some \(U \subseteq S\) containing all strata, and \(\tilde{\Delta} = f^{-1}_*
\Delta\) is snc, \(f\) may still fail to be thrifty. Unfortunately our presentation
only makes sense in characteristic 0, but I would be shocked and
appalled if this example doesn't work in any characteristic \(p>2\).

\begin{example}
  \label{ex:adm-not-thrifty}
  Let \(S = \AA^3_{xyz}\), let \(\Delta = V((z-x)(z+x))\) and let \(Z = V(x,
  y)\); let \(U = S \setminus Z\). Then there is a \(U\)-admissible blowup \(f:
  X \to S \) such that \( f^{-1}_* \Delta\) is a simple normal
  crossing divisor but \(f\) is not thrifty. 
  
  We first blow up \(Z\) to obtain \(g: \Bl_Z S \to S\), and claim that the
  strict transform of \(\Delta\) is no longer snc. Letting \(D_{\pm} =V(z\pm
  x)\) we can work in blowup coordinates described like
  \[ \Bl_Z S = \{((x, y, z) , [u, v]) \in \AA^3 \, | \, (x, y) \propto (u, v)\}
  \]
  so that on the \(D(u)\) patch \((x, y ) = \lambda (1, v)\) and 
  \[ z \pm x =  z \pm \lambda, \text{ so in } (z, \lambda, v ) \text{
  coordinates } \tilde{D}_{\pm} \cap D(u) = V(z \pm \lambda) \]
  in other words \(\tilde{\Delta}\) \emph{is} snc on the \(D(u)\) patch (as is
  expected since on \(D(x)\subseteq \AA^3\), \(\Delta\) is smooth). But on
  the \(D(v)\) patch where \((x, y) = \lambda (u, 1)\),
  \begin{equation}
    \label{eq:1st-blp}
    z \pm x =  z \pm \lambda u, \text{ so in } (z, \lambda, u )
    \text{ coordinates } \tilde{D}_{\pm} \cap D(v) = V(z \pm \lambda u)
  \end{equation}
  and here we see the strict transforms intersect along \(V(\lambda u)\) and
  hence fail to be snc 
  .


  A global description of the situation: \(\Bl_Z S\) is isomorphic to \(\AA^1_z
  \times \Bl_0 \AA^2_{xy} \), and \(\tilde{D}_{\pm}\) are 2 copies of \(\Bl_0
  \AA^2_{xy}\) embedded via the maps 
  \[ (\pm x, \mathrm{id}) : \Bl_0 \AA^2_{xy} \to  \AA^1_z \times \Bl_0
  \AA^2_{xy} \] where the map \(\pm x: \Bl_0 \AA^2_{xy} \to \AA^1_z\) really
  means the composition \(\Bl_0 \AA^2_{xy} \to \AA^2_{xy} \xrightarrow{\pm x}
  \AA^1_z\). From this perspective \(\tilde{D}_{+} \cap \tilde{D}_{-} \) is the
  preimage of \(V(x)\) under the blowup map \(\Bl_0 \AA^2_{xy} \to \AA^2_{xy}\),
  the union \(\PP^1_{xy} \cup \AA^1_y\) glued along the points \([0,1] \in
  \PP^1_{xy}\) and \( 0 \in \AA^1_y\). Let \(p \) denote the point in
  \(\PP^1_{xy} \cap \AA^1_y\). Equivalently \(\Sing (\tilde{D}_{+} \cap
  \tilde{D}_{-})\) consists of a single closed point which we call \(p\).

  This discussion shows that the snc locus of \((\Bl_Z S, \tilde{\Delta})\) is 
  \[\snc (\Bl_Z S, \tilde{\Delta}) =  \Bl_Z S \setminus \{p\}\] By work of
  Szab\'o and Bierstone-Milman \cite{MR1322695,MR1440306} (this is where we use
  the characteristic \(0\) hypothesis) there exists a further blowup \(h: X \to
  \Bl_Z S\) such that \(h^{-1}_* \tilde{\Delta} + \Ex h  \) is a simple normal
  crossing divisor \emph{and} \(h\) is an isomorphism over \(\snc (\Bl_Z S,
  \tilde{\Delta})\), that is, \(h\) must be a \(\snc (\Bl_Z S,
  \tilde{\Delta})\)-admissible blowup. Now by \cite[Thm. II.7.17]{MR0463157} we
  know that \(f:= g \circ h: X \to S\) is a blowup at \emph{some} closed
  subscheme \(W \subseteq S\) and since \(g(p) \in Z\) (equivalently)
  \(g^{-1}(U) \subseteq \snc (\Bl_Z S, \tilde{\Delta})\), it must be that \(W
  \subseteq Z\) \emph{as closed sets} (see also \cite[Lem. 5.1.4]{MR308104}),
  hence \(f: X \to S\) is a \(U\)-admissible blowup. 
  
  On the other hand, by a proposition of Erickson \cite[Prop.
  1.4]{Erickson2014DeformationIO}, since \(h^{-1}_* \tilde{\Delta} + \Ex h\) is
  snc the map \(h\) is thrifty and so  the strata of \(f^{-1}_* \Delta =
  h^{-1}_* \tilde{\Delta}\) are in 1-1 birational correspondence with those of
  \(\tilde{\Delta}\), in particular \(f^{-1}_* \Delta\) has a stratum in \(\Ex
  f\).
  
  While the application of \cite{MR1322695,MR1440306} is heavy-handed for this
  toy example, we point out that \(h\) is not simply the blowup at \(p\) as one
  might initially guess: starting from \eqref{eq:1st-blp}, blowing up the origin
  \( 0 \in \AA^3_{z\lambda u}\) and introducing blowup coordinates
  \[ \Bl_0 \AA^3_{z\lambda u} = \{((z,\lambda, u), [r,s,t]) \in \AA^3_{z\lambda
  u}  \times \PP^2_{rst} \, | \,(z,\lambda, u) \propto (r,s,t) \}  \] we note
  that since \(V((z-\lambda u) \cdot (z+\lambda u))\) is smooth on \(D(z)\) we
  can check that the strict transform remains smooth on the \(D(r)\) patch. We
  will investigate the \(D(s)\) patch --- by symmetry of \(\lambda, u\) in the
  equation \((z-\lambda u) \cdot (z+\lambda u)\) the situation is similar on the
  \(D(t)\) patch. On \(D(s)\) we have \((z,\lambda, u)= \mu (r, 1, t)\) and so 
  \[ z \pm \lambda u = \mu r \pm \mu^2 t = \mu (r \pm \mu t) \] Here \(V(\mu)\)
  is a copy of the exceptional divisor of \(\Bl_0 \AA^3_{z\lambda u} \to
  \AA^3_{z\lambda u}\) but we are still left with strict transforms \((r \pm \mu
  t)\) of exactly the same form as \(z \pm \lambda u\); in other words, blowing
  up \( 0 \in \AA^3_{z\lambda u}\) does not help! This is quite similar to the
  classical fact that blowing up the origin of the pinch point \(V(z^2 - \lambda
  u^2 ) \subseteq \AA^3_{z\lambda u}\) gives another pinch point singularity. In
  fact, since \((z-\lambda u) \cdot (z+\lambda u) = z^2 - \lambda^2 u^2\) our
  example is a double cover of the pinch point (that is, it is the preimage of
  the pinch point with respect to \((z,\lambda, u) \mapsto (z,\lambda^2, u)\)).

\end{example}

\end{document}